\documentclass[11pt]{article}
\usepackage[ruled,vlined,linesnumbered]{algorithm2e}
\usepackage{amsmath,amsfonts,amssymb,amsthm,latexsym}
\usepackage{epsfig}
\usepackage{color}
\usepackage{enumerate}
\usepackage{hyperref}
\usepackage{enumerate}
\usepackage{graphicx}
\usepackage[dvipsnames]{xcolor}
\usepackage[cp1250]{inputenc}
\usepackage{tikz}
\usepackage{tkz-graph}
\usepackage{tkz-berge}
\usepackage{subcaption}
\usepackage{float}
\usepackage{xcolor}
\usepackage{listings}
\usepackage{booktabs}

\newcommand{\smallqed}{{\tiny ($\Box$)}}

\newcommand{\dchiS}{{\rm dom}_{\chi_S}}
\newcommand{\dchir}{{\rm dom}_{\chi_\rho}}

\newcommand{\dom}{{\rm dom}}

\hypersetup{colorlinks=true, linkcolor=blue, citecolor=blue,
urlcolor=blue}
\usetikzlibrary{calc}

\hypersetup{colorlinks=true}

\hypersetup{colorlinks=true, linkcolor=blue, citecolor=blue,urlcolor=blue}

\hypersetup{colorlinks=true}

\hypersetup{colorlinks=true, linkcolor=blue, citecolor=blue,urlcolor=blue}

\hypersetup{colorlinks=true, linkcolor=blue, citecolor=blue,
	urlcolor=blue}
\usetikzlibrary{calc}

\hypersetup{colorlinks=true}

\hypersetup{colorlinks=true, linkcolor=blue, citecolor=blue,urlcolor=blue}

\hypersetup{colorlinks=true}

\hypersetup{colorlinks=true, linkcolor=blue, citecolor=blue,urlcolor=blue}

\title{Partitioning an $S$-packing coloring into broadcast dominating sets}
\date{}
\author{{Bo\v{s}tjan Bre\v{s}ar$^{a,b}$, Jasmina Ferme$^{a,c}$, and Wenjie Hu$^{a,d}$}\thanks{Corresponding author}\vspace{2mm}\\
$^{a}$Faculty of Natural Sciences and Mathematics, University of Maribor, Slovenia\vspace{1mm}\\
$^{b}$Institute of Mathematics, Physics and Mechanics, Ljubljana, Slovenia\vspace{1mm}\\
$^{c}$Faculty of Education, University of Maribor, Slovenia\vspace{1mm}\\
$^{d}$Hubei Key Laboratory of Applied Mathematics, Faculty of Mathematics and Statistics\\
Hubei University, Wuhan 430062, China\vspace{1.5mm}\\
\texttt{bostjan.bresar@um.si}\\
\texttt{jasmina.ferme1@um.si}\\
\texttt{huwj@stu.hubu.edu.cn}}
\date{}

\newtheorem{theorem}{Theorem}[section]
\newtheorem{corollary}[theorem]{Corollary}
\newtheorem{lemma}[theorem]{Lemma}
\newtheorem{observation}[theorem]{Observation}
\newtheorem{proposition}[theorem]{Proposition}

\newtheorem{claim}{Claim}
\newtheorem{problem}{Problem}

\theoremstyle{definition}
\newtheorem{definition}[theorem]{Definition}
\theoremstyle{remark}

\begin{document}

\maketitle

\begin{abstract}
Given a graph $G$ and a positive integer $k$, a packing $k$-domatic coloring of $G$ is a function $f: V(G) \to \{1,\ldots,t\}$ such that (1) for any color $j \in \{1,\ldots, t\}$ and any two distinct vertices $u,v\in V(G)$ with $f(u)=f(v)=j$ we have $d_G(u,v)>j$, and (2) $V(G)$ admits a partition into $k$ sets $A_1,\ldots,A_k$ such that for any $v\in V(G)$ and any $i\in\{1,\ldots,k\}$ there exists a vertex $w\in A_i$ such that $d_G(v,w)\le f(w)$. The minimum integer $t$ such that $G$ admits a packing $k$-domatic coloring of $G$ using colors in $\{1,\ldots,t\}$ is denoted by $\chi_{\rho,k}(G)$. The condition (1) in the definition of a packing $k$-domatic coloring implies that $G$ is a packing coloring, hence $\chi_{\rho,k}(G)\ge \chi_\rho(G)$ holds for any graph $G$, where $\chi_\rho(G)$ is the packing chromatic number of $G$. On the other hand, the new concept also leads to a generalization of the domatic number of a graph due to which one can easily see that $\chi_{\rho,2}(G)=\chi_\rho(G)$ holds for any graph with no isolated vertices. One of the main result in this paper is that $\chi_{\rho,3}(G)=\chi_\rho(G)$ holds in any connected graph $G$ with minimum degree at least $2$, and the bound is best possible in two different senses. In addition, we provide several exact values and bounds on the new invariant in paths and cycles. We prove that $\chi_{\rho,k}(P_n)=k$ for any $k\in\{3,4,5\}$ as soon as $n\ge 8$, and, in contrast, $\chi_{\rho,k}(P_n)
>k$ for any $n\ge k\ge 12$. We also prove the exact values of $\chi_{\rho,k}(P_\infty)$ when $k\in \{3,4,5\}$ for the two-way infinite path $P_\infty$, and exact values of $\chi_{\rho,k}(C_n)$ when $k\in\{3,4\}$ for all cycles $C_n$. We also consider a general framework of $S$-packing $k$-domatic colorings, where $S$ is an arbitrary sequence of non-negative integers, and present some basic results in this context. 

\end{abstract}



\section{Introduction} 

We introduce a novel graph-theoretic framework that simultaneously merges distance-based partition constraints ($S$-packing colorings~\cite{god-2012,god-2014}, initially introduced under the name broadcast colorings~\cite{god2008}) and multi-layered global coverage criteria (broadcast domination~\cite{dun2006}). One can easily observe that an arbitrary S-packing coloring can be partitioned into (at least) two dominating broadcasts by using the well known result that every vertex set of an isolate-free graph can be partitioned into two dominating set~\cite{cock1977,zel-1981}. Thus, the natural question appears into how many dominating broadcasts can an $S$-packing coloring be partitioned.  Another question, which is somehow dual to the first one, is to determine the minimum number of colors in an $S$-packing coloring for which a partition into the given number of dominating broadcasts is possible. These questions make rise to two graph invariants, which are the central theme of this work. A similar approach in standard domination theory has led to the concepts of domatic partition and domatic number, which are in a sense generalized by the framework proposed in this paper. 


\subsection*{Broadcast domination and broadcast domatic partitions}

Let $G = (V(G), E(G))$ be a simple, connected, undirected graph.  Denote by $n(G)$ the order of $G$. Let $d_G(u, v)$ denote the shortest-path distance between vertices $u$ and $v$ in $G$, and let ${\rm diam}(G)$ be the {\em diameter} of $G$, which is defined as $\max\{d_G(u,v):\; u,v\in V(G)\}$. In this paper, we use the notation $[n]=\{1,\ldots,n\}$ and $[n]_0=\{0\}\cup [n]$. 
For a vertex $v\in V(G)$, the {\em neighborhood}, $N_G(v)$, of $v$ is the set of vertices adjacent to $v$, while the {\em closed neighborhood} of $v$ is defined as $N_G[v]=N_G(v)\cup\{v\}$. If $S\subset V(G)$, then let $N_G(S)=\bigcup_{v\in S}{N_G(v)}$. If the graph is understood from the context, then index $G$ in the above definitions may be omitted. Further, if $u$ is a a vertex of a graph $G$, and $k$ is a positive integer, then $N_k(u)=\{x\in V(G):\, d(x,u)\le k\}$.  
The {\em degree}, $\deg(v)$, of $v\in V(G)$ is defined as $|N_G(v)|$, while $\delta(G)=\min\{\deg(v):\,v\in V(G)\}$. 

A subset of vertices $D \subseteq V$ is a \textit{dominating set} of $G$ if every vertex $v \in V(G) \setminus D$ is adjacent to at least one vertex in $D$. In other words, for the closed neighborhood of every vertex $v$ in $G$, $N_G[v] \cap D \neq \emptyset$. The minimum cardinality of a dominating set in $G$ is the \textit{domination number}, $\gamma(G)$, of $G$.

A \textit{broadcast} on a graph $G$ is a function $f: V \to \mathbb{N}_0$
that assigns a non-negative integer to each vertex. 
A broadcast $f$ is a \textit{dominating broadcast} on $G$ if for every vertex $u \in V(G)$, there exists a vertex $v \in V(G)$ (possibly $u = v$) such that $f(v) > 0$ and 
\begin{equation}
\label{eq:dombroad}
d_G(u, v) \le f(v).    
\end{equation}
In other words, every vertex in the graph must be within the distance $f(v)$ of at least one broadcasting vertex $v$ with $f(v)>0$. The concept of broadcast domination was first introduced by Erwin in his Ph.D. dissertation~\cite{er2001}. The foundational properties and formal framework were subsequently published by Dunbar, Haynes, Hedetniemi, Hedetniemi, and McRae in~\cite{dun2006}. The authors studied the question of a minimum cost of a dominating broadcast, called the broadcast domination number of graph. In this paper, we are interested in a partition of the vertex set into dominating broadcasts, which is similar to the concept of domatic partition (see Section~\ref{sec:prel} for more detail on this connection).

Given a broadcast $f: V \to \mathbb{N}_0$,  
a \textit{broadcast domatic $k$-partition} of a graph $G$ {\em with respect to $f$} is a partition  $\pi = \{V_1, \dots, V_k\}$  of the vertex set $V(G)$ such that for each subset $V_i$, where $1 \le i \le k$, the function $f|_{V_i}$ ($f$ restricted to $V_i$) is a dominating broadcast of $G$. In other words, letting $f_i:V(G)\to \mathbb{N}_0$ be defined as 
$$f_i(v) =  \left \{
    \begin{array}{ll}
	  \!\!\!f(v) &\text{if~} v\in V_i,\\
	0 \hskip0.2cm &\text{otherwise},\\
    \end{array}
    \right.$$ 
$f_i$ is a dominating broadcast for every $i\in [k]$. Given $f:V\to \mathbb{N}_0$ if there exists a broadcast domatic $k$-partition of $G$ with respect to $f$, then $f$ is a {\em $k$-dominating broadcast of $G$}.

\subsection*{$S$-packing (domatic) coloring}

The concepts of broadcast coloring and broadcast chromatic number were introduced in 2008 by Goddard et al.~\cite{god2008}. In all subsequent papers, the term ``broadcast'' in the name was replaced by ``packing'', hence broadcast coloring was called packing coloring, which is the name we will use also in this paper. Packing coloring is a vibrant topic of research; see a survey~\cite{BFKR} from 2020, and some more recent references~\cite{Kostochka-2021,liu-2025,mt-2024, tt-2024, YW}. 

Let $S=(s_1,s_2,\ldots)$ be a non-decreasing sequence of positive integers. A function $f: V(G) \rightarrow \mathbb{N}$ is an {\em $S$-packing coloring} of a graph $G$ if for any distinct vertices $u,v\in V(G)$ with $f(u)=f(v)=i$, we have $d_G(u,v)>s_i$. In other words, any two vertices assigned the exact same color $i$ must be at a distance strictly greater than $s_i$ from each other. Equivalently, each color class $P_i = \{v \in V :\, f(v) = i\}$ forms an {\em $i$-packing} in $G$.

The smallest integer $t$ for which $G$ admits an $S$-packing coloring whose largest color is $t$ is the {\em $S$-packing chromatic number} of $G$, denoted $\chi_S(G)$. If $\chi_S(G)\le \ell$ we say that $G$ is {\em $(s_1,\ldots,s_\ell)$-packing colorable}. Hence, $S$-packing coloring has a meaning also when $S$ is a finite sequence, in which case we are not dealing with optimization problem (finding the minimum number of colors), but with the decision problem whether $G$ admits an $S$-packing coloring or not. 
In the special case when $S=(n)_{n\in \mathbb{N}}$ is the sequence of positive integer, an $S$-packing coloring is simply called packing coloring and the resulting invariant is the {\em packing chromatic number} of $G$, denoted by $\chi_\rho(G)$.

Now, we are ready to present definitions of the main concepts that we introduce in this paper. They are based on combining the concepts of broadcast domatic partitions and $S$-packing colorings. 
Let $G$ be a graph, $S$ a non-decreasing sequence of positive integers  and $k$ a positive integer. An \textit{$S$-packing $k$-domatic coloring} of $G$ is an $S$-packing coloring $f: V(G) \to \mathbb{N}$, which is at the same time a $k$-dominating broadcast. That is, there exists a partition of $V(G)$ into $k$ subsets $A_1, \ldots, A_k$ such that each $f|_{A_i}$ is a dominating broadcast of $G$. In particular, for every vertex $x \in V(G)$ and every set $A_i$ ($i\in [k]$) there exists a vertex $v \in A_i$ such that $d(x, v) \le f(v)$. In such a case, we also say that a vertex $x$ is $f$-dominated by the sets $A_1,\ldots,A_k$. 

\subsection*{New graph invariants}

Two optimization concepts naturally arise from the above definition of $S$-packing $k$-domatic coloring. 

\begin{definition}
\label{def:2}
    Let $G$ be a graph, $S=(s_1,s_2,\ldots)$ a non-decreasing sequence of positive integers and $k$ a positive integer, where $1\le k\le n(G)$. The minimum $t$ for which there exists an $S$-packing $k$-domatic coloring of $G$ using colors in $[t]$ is an \textit{$S$-packing $k$-domatic chromatic number} of $G$, denoted by $\chi_{S,k}(G)$. If there is no such $t\in \mathbb{N}$, we set $\chi_{S,k}(G)=\infty$.
\end{definition}

The concept in Definition~\ref{def:2} is well defined. Indeed, if $G$ is an arbitrary graph and $S$ an arbitrary non-decreasing sequence of positive integers and $k\ge 1$, then either there is some $t\in \mathbb{N}$ for which there exists an $S$-packing $k$-domatic coloring of $G$ using colors in $[t]$ or not. In the first case, there also exists the smallest $t$ satisfying this property, which then yields the value of $\chi_{S,k}(G)$. In the second case, $\chi_{S,k}(G)=\infty$ by definition.  

In the special case when $S=(n)_{n\in\mathbb{N}}$, we omit $S$ in the name of the invariant and speak of \textit{packing $k$-domatic chromatic number} of $G$ and denote it by $\chi_{\rho,k}(G)$. 

\begin{definition}
\label{def:3}
Let $G$ be a graph and $S=(s_1,\ldots,s_t)$ a finite non-decreasing sequence of positive integers such that $G$ is $S$-packing colorable. The maximum $k$ for which there exists an $S$-packing $k$-domatic coloring of $G$ is the {\em $S$-packing domatic number} of $G$, denoted $\dom_{\chi_S}(G)$. 
    
If $S=(s_1,s_2,\ldots)$ is an infinite non-decreasing sequence of positive integers, then let $\dom_{\chi_S}(G)=dom_{\chi_{S'}}(G)$, where $S'=(s_1,\ldots,s_{\chi_S(G)})$. That is, we consider the first $\chi_S(G)$ terms in the sequence $S$ using the definition for a finite sequence.  
\end{definition}

To see that the concept in Definition~\ref{def:3} is well defined, let $G$ be a graph and $S=(s_1,\ldots,s_t)$ be a finite non-decreasing sequence $S$ of positive integers such that $G$ is $S$-packing colorable. Taking $k=1$, note that any $S$-packing coloring $f$ of $G$ (there exists at least one such coloring by the assumption) is an $S$-packing $1$-domatic coloring of $G$. Indeed, it is obvious that $f$ is a $1$-dominating broadcast of $G$ since every vertex $v\in V(G)$ $f$-dominates itself; notably, $d_G(v,v)=0<1\le f(v)$. Hence, in any graph $G$ which is $S$-packing colorable, $\dom_{\chi_S}(G)\ge 1$. 

In the special case when $S=(n)_{n\in\mathbb{N}}$, the invariant in Definition~\ref{def:3}, is denoted by $\dom_{\chi_\rho}(G)$, and we call it the packing domatic number of $G$.

\subsection*{Conceptual remarks and potential applications}

The primary goal of this work is to investigate the combinatorial properties of the newly introduced invariants and their relationship to existing concepts from domination and packing coloring theory. On the other hand, the concepts may provide several applications, which we present mainly as sources of intuition and motivation.  

The main interest in $S$-packing $k$-domatic colorings stems from the interaction between two competing requirements. On the one hand, vertices assigned the same value must be sufficiently separated, as prescribed by the $S$-packing condition. On the other hand, the same values determine the broadcast ranges that are used to achieve domination. Larger values therefore provide greater coverage but are more difficult to place due to the packing restrictions. This interplay between separation and coverage appears to be a natural extension of both packing colorings and broadcast domination.

Possible interpretations include the following. \textbf{Communication networks.}
    Vertices may represent transmitters with different transmission ranges. The packing condition prevents transmitters of the same type from being placed too close to one another, while each layer of the partition provides complete coverage of the network. \textbf{Service and facility networks.} Vertices may represent service centers or facilities with different service radii. The model captures the requirement that facilities of the same category are sufficiently dispersed, while each layer remains capable of serving the entire network. \textbf{Distributed systems.} Vertices may represent agents, processors, or data centers with different communication capabilities. The partition into several dominating layers can be interpreted as providing multiple independent structures capable of reaching all parts of the network.

\subsection*{Organization of the paper and main results}

In Section~\ref{sec:prel}, we establish several preliminary properties of the two newly introduced invariants. In particular, we prove monotonicity of $\chi_{S,k}(G)$ with respect to $k$ and describe a duality between the $S$-packing $k$-domatic chromatic number and the $S$-packing domatic number: $$\chi_{S,k}(G)=\chi_S(G) \quad\text{if and only if}\quad \dom_{\chi_S}(G)\geq k.$$ We also explain how the classical domatic number arises as a special case of the new framework. For graphs of diameter~$2$ and sequences satisfying $s_1=1$ and $s_2\geq 2$, we determine the exact value $\operatorname{dom}_{\chi_S}(G) =n(G)-\alpha(G)+1 =\beta(G)+1$. Besides illustrating the connection between packing colorings, domination, and classical covering parameters, this formula yields NP-completeness of the decision version of $\operatorname{dom}_{\chi_S}$ even for graphs of diameter~$2$. We further show that the decision version of $\chi_{\rho,2}$ remains NP-complete when restricted to trees. These results indicate the difficulty of the newly introduced invariants, and explain our decision to study it in some basic families of graphs. In particular, paths and cycles are one of the rare graph classes in which the packing chromatic number is easy to obtain, hence our decision to study their packing $k$-domatic chromatic numbers, which is done in Section~\ref{sec:paths}.

For finite paths, we first establish the general lower bound $\chi_{\rho,k}(P_n)\geq k$ for all $k\geq 3$ and $n\geq 3$. We then determine the exact values for $k\in\{3,4,5\}$. In particular, $\chi_{\rho,3}(P_n)=3$ for every $n\geq 3$, $\chi_{\rho,4}(P_n)=4$ for every $n\geq 5$, with $\chi_{\rho,4}(P_4)=5$, and $\chi_{\rho,5}(P_n)=5$ for every $n\geq 8$, while $\chi_{\rho,5}(P_n)=6$ for $n\in\{5,6,7\}$. Thus, for each $k\in\{3,4,5\}$, the lower bound $\chi_{\rho,k}(P_n)\geq k$ is attained for all sufficiently large $n$. In contrast, we prove that $\chi_{\rho,k}(P_n)>k$ whenever $k\geq 12$ and $n\geq k$. This reveals a qualitative change in the behavior of the invariant as the required number of dominating broadcasts increases. For the two-way infinite path $P_\infty$, we determine $$\chi_{\rho,3}(P_\infty) =\chi_{\rho,4}(P_\infty)=3 \quad\text{and}\quad \chi_{\rho,5}(P_\infty)=4, $$ and show that $ \chi_{\rho,6}(P_\infty)\in\{4,5\}$. More generally, we prove the lower bound $ \chi_{\rho,k}(P_\infty)>\frac{k-1}{2}$ for every $k\geq 3$. These results demonstrate that the finite and infinite versions of the path problem exhibit substantially different behavior. We also determine $\chi_{\rho,k}(C_n)$ for $k\in\{3,4\}$ and all cycles. For $k=3$, we prove that $$\chi_{\rho,3}(C_n)= \begin{cases} 3 &\text{if }  n\equiv 0\pmod 4,\\ 4 & \text{if } n\not\equiv 0\pmod 4, \end{cases}$$ for $n\geq 4$, together with $\chi_{\rho,3}(C_3)=3$. For $k=4$, the value is $5$ for $C_4$, it is $4$ whenever $n\not\equiv0\pmod4$, and it is $3$ whenever $n\equiv0\pmod4$, except for the two exceptional cycles $C_8$ and $C_{20}$, for which the value is $4$. The lower bound for $C_{20}$ is verified by an exhaustive computational search, whose code is included in Appendix~A. 

In Section~\ref{sec:domatic}, we prove one of the main structural results of the paper: $$\dom_{\chi_\rho}(G)\geq 3$$ for every connected graph $G$ with minimum degree at least~$2$. Equivalently, $\chi_{\rho,3}(G)=\chi_\rho(G)$ for every such graph. This strengthens the classical fact that every graph without isolated vertices admits a partition into two dominating sets and shows that the additional flexibility provided by broadcast ranges guarantees a third dominating layer under the natural condition $\delta(G)\geq2$. We show that the bound is sharp, both by exhibiting graphs with minimum degree~$2$ and packing domatic number~$3$, and by giving examples with minimum degree~$1$ and packing domatic number~$2$. As a complexity consequence, we prove that the decision version of $\chi_{\rho,3}$ is NP-complete even for graphs of diameter~$2$ and minimum degree at least~$2$. 
Finally, Section~\ref{sec:conclude} collects several open problems arising from our results.  We also indicate that the broader study of $S$-packing domatic colorings for sequences other than $(1,2,3,\ldots)$ remains largely open.

\section{Preliminary observations}
\label{sec:prel}

First, we notice the monotonicity property with respect to $k$ for the parameter $\chi_{S,k}$.  It immediately follows from the fact that if $G$ admits an $S$-packing $k$-domatic partition for some $k\ge 2$, then by merging any two sets from the partition we get an $S$-packing $(k-1)$-domatic partition (using the same $S$-packing coloring $f$).
\begin{observation}
\label{obs:monotonicity}
    If $S$ is a sequence of positive integers and $G$ is a graph with $\dchiS(G)\ge k\ge 2$, then 
$$\chi_{S,k-1}(G)\le \chi_{S,k}(G).$$ 
\end{observation}

It is easy to see that for some sequences $S$ and some positive integers $k$, there are graphs $G$ for which $\chi_{S,k}(G)=\infty$. For instance, if $S=(1,1,\ldots)$ is the constant sequence of $1$s, then the existence of an $S$-packing $k$-domatic colorings depends only on the question whether $G$ can be partitioned into $k$ dominating sets. 
However, if the sequence $S$ is unbounded, the following observation holds.

\begin{observation}
If $S$ is an unbounded non-decreasing sequence of positive integers, then $\chi_{S,k}(G)<\infty$ holds for any connected graph $G$ and any $k\le n(G)$. \end{observation}
\begin{proof}
Let $S$ be an unbounded non-decreasing sequence of positive integers, and let $i\in \mathbb{N}$ be the smallest index such that $s_i\ge{\rm diam}(G)$. Note that for any $v\in V(G)$, letting $f(v)=s_i$ and $f(u)=0$ for all $u\ne v$ gives a dominating broadcast of $G$. Thus, assigning pairwise distinct values from $\{s_i,s_{i+1},\ldots,s_{i+n-1}\}$ to the vertices of $G$, where $n=n(G)$, yields an $n$-dominating broadcast of $G$, which is at the same an $S$-packing coloring. Therefore, $$\chi_{S,n}(G)\le i+n(G)-1,$$
and by Observation~\ref{obs:monotonicity}, we infer that $\chi_{S,k}(G)<\infty$ for all $k\le n(G)$. 
\end{proof}

Next, we present a relation between two concepts from Definitions~\ref{def:2} and~\ref{def:3}, which provides a version of duality between them.

\begin{observation}
\label{ob:domandchi}
Let $G$ be a graph, $S=(s_1,s_2,\ldots)$ a non-decreasing sequence of positive integers and $k$ a positive integer. Then, $\chi_{S,k}(G)=\chi_S(G)$ if and only if $\dom_{\chi_S}(G)\ge k$.
\end{observation}
\begin{proof}
Let $t=\chi_{S,k}(G)<\infty$. Thus there is an $S$-packing $k$-domatic coloring $f$ of $G$ using colors in $[t]$. If $t=\chi_S(G)$, we derive that $f$ is an $S$-packing coloring of $G$ using $\chi_S(G)$ colors, which is at the same time a $k$-dominating broadcast. Thus, $\dom_{\chi_S}(G)\ge k$. 

 For the converse, note that since $S$ is infinite, $\dom_{\chi_S}(G)$ is defined as $\dom_{\chi_{S'}}(G)$, where $S'=(s_1,\ldots,s_{\chi_S(G)})$. Let $r=\dom_{\chi_{S'}}(G)$, and note that $r\ge k$. By definition, there exists an $S$-packing $r$-domatic coloring of $G$ using $\chi_S(G)$ colors, thus $\chi_{S,r}(G)=\chi_S(G)$. By applying Observation~\ref{obs:monotonicity}, we infer that $\chi_{S,k}(G)\le \chi_{S,r}(G)=\chi_S(G)$, which gives the desired equality. 
\end{proof}

The concepts of domatic partition and domatic number were first introduced by Cockayne and Hedetniemi in 1977 in their seminal paper~\cite{cock1977}, and were followed by a number of studies of these concepts and their variants; see the monographs on graph domination~\cite{HaHeHe-20, HaHeHe-21,HaHeHe-23}. 
A \textit{domatic partition} of a graph $G$ is a partition of the vertex set $V(G)$ into disjoint subsets $\pi = \{V_1,  \dots, V_k\}$ such that
each subset $V_i$ (for $1 \le i \le k$) is a dominating set of $G$.
The \textit{domatic number} of a graph $G$, denoted by $\dom(G)$, is the maximum number of sets in a domatic partition of $G$, that is,
$
\dom(G) = \max \{ |\pi| :\, \pi \text{ is a domatic partition of } G \}.
$
It is a well known classical result that $\dom(G)\ge 2$ in every graph $G$ without isolated vertices, while $\dom(G)\le \delta(G)+1$ is a trivial upper bound on the domatic number of a graph; see~\cite{GH2021,HaHeHe-21}. 
Thus, one can partition such a graph into two dominating sets, and the same partition can be used also in any $S$-packing $2$-domatic coloring.

\begin{observation}
\label{obs:domatic}
If $G$ is a graph with no isolated vertices and $S=(s_1,s_2,\ldots)$ an infinite non-decreasing sequence of positive integers, then $$\dom_{\chi_S}(G)\ge 2.$$    
\end{observation}

Combining Observations~\ref{ob:domandchi} and~\ref{obs:domatic} we infer that for any graph $G$ with no isolated vertices and any infinite non-decreasing sequence of positive integers $S$, $$\chi_{S,2}(G)=\chi_S(G).$$ 

Thus, Observation~\ref{obs:domatic} presents a trivial generalization of the well known fact about the partition of the vertex set of an isolate-free graph into two dominating sets. In this paper, we will extend this result by showing that the vertex set of a graph $G$ with $\delta(G)\ge 2$ can be partitioned into more than two sets by using the setting of $S$-packing domatic colorings.  

As indicated earlier, the $S$-packing domatic number indeed presents a vast generalization of the domatic number. Notably, if $S$ is the infinite sequence of all $1$s, then the $S$-packing domatic number does not depend on the coloring (which is just a proper coloring of $G$), but solely on the number of dominating sets into which $G$ can be partitioned. 

\begin{observation}
\label{obs:domgen}
If $G$ is a graph with no isolated vertices and $S=(1,1,\ldots)$ is the sequence with all terms equal to $1$, then $$\dom_{\chi_S}(G)=dom(G).$$    \end{observation}

Next, we focus on diameter-2 graphs and sequence with only one integer $1$. Recall that $\alpha (G)$, resp.\ $\beta(G)$, is the {\em independence number}, resp.\ {\it the vertex cover number} of $G$, defined as the maximum number of pairwise non-adjacent vertices, resp.\ the minimum number of vertices that are incident with every edge of $G$.  By Gallai's theorem $\alpha(G)+\beta(G)=n(G)$ in any graph $G$; cf.~\cite{West}. 
\begin{proposition}
\label{prp:diam2}
If $G$ is a graph with ${\rm diam}(G)=2$ and $S$ is a non-decreasing sequence of positive integers with $s_1=1$ and $s_2\ge 2$, then $\dom_{\chi_S}(G)=n(G)-\alpha(G)+1=\beta(G)+1$. 
\end{proposition} 
\begin{proof}
It is easy to see that $\chi_S(G)=n(G)-\alpha(G)+1$. Let $f$ be the corresponding $S$-packing coloring. Now, consider the partition $\pi=\{D_1,\ldots,D_{n-\alpha(G)+1}\}$ of $V(G)$, where $D_1$ is a maximum independent set of $G$ and $D_i$, where $i\ge 2$, is a singleton. Clearly, every set $D_i$ $f$-dominates $G$, hence $f$ is a packing $(n-\alpha+1)$-domatic coloring, thus $\dom_{\chi_S}(G)\ge n-\alpha+1$. It is easy to see that $\dom_{\chi_S}(G)$ cannot be bigger than $n(G)-\alpha(G)+1$.  
\end{proof}

It is well known that the decision version of the independence number is NP-complete~\cite{gj}, and it clearly remains NP-complete when restricted to diameter-2 graphs.

\begin{corollary}
Let $S$ be a non-decreasing sequence of positive integers with $s_1=1$ and $s_2\ge 2$. Then, the decision version of $\dom_{\chi_S}$ is NP-complete even when restricted to graphs with diameter $2$. 
 \end{corollary}

Combining Observations~\ref{ob:domandchi} and~\ref{obs:domatic} we infer that $\chi_{S,2}(G)=\chi_S(G)$ holds for any graph without isolated vertices and any infinite non-decreasing sequence $S$ of positive integers. As a consequence of this equality, we obtain a complexity result concerning the sequence $S=(1,2,3,\ldots)$, by using also that the decision version of the packing chromatic number is NP-complete even in the class of trees~\cite{fg-2010}. 

\begin{corollary}
The decision version of $\chi_{\rho,2}$ is NP-complete even when restricted to trees. 
\end{corollary}

We mention that Corollary~\ref{cor:completexity3} will provide  NP-completeness for $\chi_{\rho,3}$ restricted to a different family of graphs (namely, graphs with diameter $2$ and $\delta(G)\ge 2$).

\section{The invariant $\chi_{\rho,k}$ in paths and cycles}
\label{sec:paths}

In this section, we determine exact values or bounds on the packing $k$-domatic chromatic numbers for (finite and infinite) paths and cycles for several values of $k$. The most extensive are our investigations in finite paths, where we obtain  bounds on $\chi_{\rho, k}(P_n)$ for an arbitrary value of $k$ and $n$.

\subsection{Finite paths}

In this subsection, the consecutive vertices of the path $P_n$, where $n\ge3$, are denoted by $u_1,u_2,\ldots,u_n$.
Recall that $\chi_\rho(P_3)=2$ and $\chi_\rho(P_n)=3$ for every $n\ge4$. Consequently, $\chi_{\rho,k}(P_3)\ge2$ and $\chi_{\rho,k}(P_n)\ge3$ for every positive integer $k$ and $n \geq 4$. The following  auxiliary result establishes a general lower bound on $\chi_{\rho,k}(P_n)$, which holds for any $n\ge 3$ and any $k\ge 3$. 

\begin{lemma}
\label{lem:paths_lower}
    If $k \geq 3$ and $n \geq 3$, then $\chi_{\rho, k}(P_n) \geq k$.
\end{lemma}
\begin{proof}
Let $k \geq 3$ and $n \geq 3$ be arbitrary integers. Suppose that $\chi_{\rho, k}(P_n) \leq k-1$ and denote by $f$ a packing $k$-domatic coloring using $k-1$ colors. Further, let $\Pi=\{A_1, A_2,\ldots,A_k\}$ be an arbitrary partition of $V(P_n)$ such that for each $i \in [k]$, $f|_{A_i}$ is a dominating broadcast of $P_n$. In particular, for each $i \in [k]$, there exists a vertex $v \in A_i$ such that $d(u_1, v) \leq f(v)$. Since $f(v) \leq k-1$ for any $v$, each of the sets $A_i$ contains at least one vertex from $N_{k-1}[u_1]=\{u_1, u_2, \ldots, u_k\}$. Since the sets in $\Pi$ are pairwise disjoint and $|\Pi|=|N_{k-1}[u_1]|=k$, each set $A_i$ contains exactly one of the vertices from $N_{k-1} [u_1]$. This implies that $f(u_k)=k-1$, $f(u_{k-1})=k-2$, $\ldots$, $f(u_2)=1$ and $f(u_1)=f(u_i)=i-1$ for some $i \in \{2,3,\ldots, k\}$, which contradicts the fact that $f$ is a packing coloring since $d(u_1, u_i)=i-1$. Thus, $\chi_{\rho, k}(P_n) \geq k$.
\end{proof}

\begin{proposition}
    \label{prop_chi_rho_3_finite_paths}
For any path $P_n$, where $n \geq 3$, 
$\chi_{\rho, 3}(P_n) =3.$
\end{proposition}

\begin{proof}
    Let $n \geq 3$ be an arbitrary integer. By Lemma~\ref{lem:paths_lower}, we have $\chi_{\rho, 3}(P_n) \geq 3$. To prove that $\chi_{\rho, 3}(P_n) \leq 3$, we form a packing $3$-domatic coloring $f$ of $P_n$ using $3$ colors as follows. If $n=3$, then let $f(u_1)=2$, $f(u_2)=1$ and $f(u_3)=3$. In the case when $n >3$ and $n \not\equiv 1 \ (\mathrm{mod} \ 4)$, color the consecutive vertices of $P_n$ one after another by applying the color pattern $[1,2,1,3]$. Finally, if $n >3$ and $n \equiv 1 \ (\mathrm{mod} \ 4)$, then let $f(u_1)=2$, $f(u_2)=1$, $f(u_3)=3$ and color the remaining vertices of $P_n$ one after another (starting with $u_4$) by repeatedly applying the color pattern $[1,2,1,3]$. Note that the above construction yields a packing coloring of $P_n$ for any $n \geq 3$. Furthermore, for each $i \in [3]$, let $A_i=\{u_j:\,f(u_j)=i\}$. Since $\Pi=\{A_1, A_2, A_3\}$ is a partition of $V(P_n)$ and for each $i \in [3]$, the restriction $f|_{A_i}$ is a dominating broadcast, we conclude that $\chi_{\rho, 3}(P_n) = 3$ for any $n \geq 3$. 
\end{proof}

\begin{proposition}
    \label{prop_chi_rho_4_finite_paths}
For any path $P_n$, where $n \geq 4$, 
$$\chi_{\rho, 4}(P_n) =
\left\{\begin{array}{ll}
5 & \text{if }  n=4\,,\\
4 & \text{if } n \geq 5\,.
\end{array}\right.$$
\end{proposition}

\begin{proof}
    First, consider the case when $n=4$. By Lemma~\ref{lem:paths_lower}, $\chi_{\rho, 4}(P_4) \geq 4$. It is clear that colors in $[4]$ do not suffice for a packing $4$-domatic coloring of $P_4$, thus $\chi_{\rho, 4}(P_4) > 4$. By setting $f(u_1)=5$, $f(u_2)=2$, $f(u_3)=3$ and $f(u_4)=4$, we obtain a packing $4$-domatic coloring $f$ of $P_4$ using colors in $[5]$, thus $\chi_{\rho, 4}(P_4) = 5$.

    Next, let $n \geq 5$. By Lemma~\ref{lem:paths_lower}, $\chi_{\rho, 4}(P_n) \geq 4$. In order to prove that $\chi_{\rho, 4}(P_n) \leq 4$, we form a packing $4$-domatic coloring $f$ of $P_n$ using $4$ colors. To present this coloring, we distinguish five cases with respect to $n$. \\
    \textbf{Case 1. $n \equiv 0 \ (\mathrm{mod}\ 5)$.} \\
    In this case, color the consecutive vertices of $P_n$ one after another by repeating the color pattern $[1, 3, 2, 1, 4]$. \\
    \textbf{Case 2. $n \equiv 1 \ (\mathrm{mod} \ 5)$.} \\
    First, let $f$ assign the colors $2, 1,3,4,1,2$ to the vertices $u_1, u_2, \ldots$ and $u_6$, respectively. Next, if $n>6$, apply the color pattern $[1, 3, 2, 1, 4]$ to color the remaining consecutive vertices of $P_n$. \\
    \textbf{Case 3. $n \equiv 2 \ (\mathrm{mod} \ 5)$.} \\
    Let the vertices $u_1, u_2, \ldots, u_7$ receive colors $1, 2, 1,3,4,1,2$, respectively. If $n>7$, color the remaining vertices of $P_n$ one after another by applying the color pattern $[1, 3, 2, 1, 4]$. \\
    \textbf{Case 4. $n \equiv 3 \ (\mathrm{mod} \ 5)$.} \\
    In this case, we first color the vertices $u_1, u_2, \ldots, u_8$ with colors $3,1,2,1,4,3,1,2$, respectively. Then, if necessary, apply the color pattern $[1, 3, 2, 1, 4]$ for coloring the remaining consecutive vertices of $P_n$. \\
    \textbf{Case 5. $n \equiv 4 \ (\mathrm{mod} \ 5)$.} \\
    Let the vertices $u_1,u_2,u_3,u_4$ receive colors $3,1,2,4$, respectively. Next, color the remaining vertices one after another using the color pattern $[1, 3, 2, 1, 4]$. 
 
  It is straightforward to verify that in each of the cases, $f$ is a packing coloring. Next, for each $i \in [4]$, let $A_i=\{u_j:\,f(u_j)=i\}$. Since $\Pi=\{A_1, A_2, A_3, A_4\}$ is a partition of $V(P_n)$ and for each $i \in [4]$, $f|_{A_i}$ is a dominating broadcast, we conclude that $\chi_{\rho, 4}(P_n) = 4$. 
\end{proof}
\begin{proposition}
    \label{prop_chi_rho_5_finite_paths}
For any path $P_n$, where $n \geq 5$, 
$$\chi_{\rho, 5}(P_n) =
\left\{\begin{array}{ll}
6 & \text{if } n \in \{5,6,7\},\\
5 & \text{if } n \geq 8\,.
\end{array}\right.$$
\end{proposition}

\begin{proof} 
First, consider the case when $n=5$.  It is clear that colors in $[5]$ do not suffice for a packing $5$-domatic coloring of $P_5$, thus $\chi_{\rho, 5}(P_5) > 5$.  
By setting $f(u_1)=5$, $f(u_2)=4$, $f(u_3)=2$, $f(u_4)=3$ and $f(u_5)=6$, we form a packing $5$-domatic coloring $f$ of $P_5$ using colors from $[6]$. Therefore, $\chi_{\rho, 5}(P_5) = 6$.

Now, let $n=6$. By Lemma~\ref{lem:paths_lower}, we have $\chi_{\rho,5}(P_6)\geq5$. Suppose that $\chi_{\rho,5}(P_6)=5$ and denote by $f$ a packing $5$-domatic coloring of $P_6$ which uses colors from $[5]$. Further, let $\Pi=\{A_1,A_2,A_3,A_4,A_5\}$ be a partition of $V(P_6)$ such that $f|_{A_i}$ is a dominating broadcast for every $i\in[5]$. Since $|V(P_6)|=6$, at least four sets in $\Pi$ contain only one element. Hence, at least four vertices of $P_6$ individually $f$-dominate all vertices of $P_6$. Since the radius of $P_6$ is $3$, each of these four vertices receives a color from $\{3,4,5\}$ by $f$. Therefore, two of these vertices receive the same color from $j \in \{3,4,5\}$. Moreover, both of them $f$-dominate $u_1$, which means that they are at distance at most $j$, a contradiction to the fact that $f$ is a packing coloring. Thus, $\chi_{\rho,5}(P_6)\geq6$. \\
Now, let $f(u_1)=1$, $f(u_2)=4$, $f(u_3)=3$, $f(u_4)=2$, $f(u_5)=5$, and $f(u_6)=6$. Furthermore, let $A_1=\{u_1,u_4\}$, $A_2=\{u_2\}$, $A_3=\{u_3\}$, $A_4=\{u_5\}$, and $A_5=\{u_6\}$. It is straightforward to verify that $f$ is a packing coloring and that $f|_{A_i}$ is a dominating broadcast of $P_6$ for every $i\in[5]$. Therefore, $\chi_{\rho,5}(P_6)\leq6$, and consequently $\chi_{\rho,5}(P_6)=6$.

Next, consider the case when $n=7$. By Lemma~\ref{lem:paths_lower}, we have $\chi_{\rho,5}(P_7)\geq5$. Suppose that there exists a packing $5$-domatic coloring $f$ of $P_7$ using colors from $[5]$. Let $\Pi=\{A_1,A_2,A_3,A_4,A_5\}$ be a partition of $V(P_7)$ such that $f|_{A_i}$ is a dominating broadcast for every $i\in[5]$. We observe that there are at least three sets in $\Pi$ such that each contains exactly one vertex of $P_7$, since $|V(P_7)|=7$. This means that $V(P_7)$ contains at least three vertices such that each of them $f$-dominates all vertices of $P_7$.
Since the radius of $P_7$ is $3$, each of these vertices must receive a color from $\{3,4,5\}$ by $f$. Note that all of these vertices $f$-dominate $u_1$, hence they receive pairwise distinct colors from $\{3,4,5\}$, which also means that we have exactly three sets in $\Pi$ such that each of them contains only one element. Therefore, we can assume that $|A_1|=1$ and the vertex in $A_1$ is colored with $3$ by $f$, $|A_2|=1$ and the vertex in $A_2$ is colored with $4$ by $f$, $|A_3|=1$ and the vertex in $A_3$ is colored with $5$ by $f$, and $|A_4|\geq 2$, $|A_5|\geq 2$. These assumptions imply that $A_1=\{u_4\}$, $A_2 \subseteq \{u_3, u_5\}$, $A_3 \subseteq \{u_2, u_3, u_5, u_6\}$. Additionally, since $f$ is a packing coloring, the vertices in $A_1 \cup A_2 \cup A_3$ are the only vertices of $P_7$ receiving colors from $\{3,4,5\}$. Consequently, the remaining four vertices from $V(P_7) \setminus (A_1 \cup A_2 \cup A_3)$ receive colors from $\{1,2\}$ and belong to $A_4 \cup A_5$. \\
By symmetry, we may assume that $A_2=\{u_3\}$, which implies that $A_3 \subseteq \{u_2, u_5, u_6\}$. If $A_3=\{u_2\}$, then the vertices $u_1,u_5,u_6,u_7$ receive colors from $\{1,2\}$ and at least two of them $f$-dominate $u_1$, which is not possible. Hence, $A_3 \neq \{u_2\}$. Now, suppose that $A_3=\{u_5\}$, which means that the vertices $u_1,u_2,u_6,u_7$ receive colors from $\{1,2\}$ and $A_4 \cup A_5 =\{u_1,u_2,u_6,u_7\}$. Since $A_4$ and $A_5$ must $f$-dominate $u_4$, we derive that $f(u_2)=f(u_6)=2$ and we can assume that $u_2 \in A_4$ and $u_6 \in A_5$. Consequently $u_1 \in A_5$ and $f(u_1)=1$, which implies that $A_5$ does not $f$-dominate $u_3$, a contradiction. Hence, $A_3=\{u_6\}$. Then neither $u_1$ nor $u_2$ $f$-dominates $u_7$, thus we can assume that $u_5 \in A_4$ and $u_7 \in A_5$, which also implies that $f(u_5)=2$ and $f(u_7)=1$. But then no vertex in $A_5$ $f$-dominates $u_5$, again a contradiction. Thus, $\chi_{\rho,5}(P_7)\geq6$. 

For the reverse inequality, define $f$ by setting $f(u_1)=1$, $f(u_2)=2$, $f(u_3)=4$, $f(u_4)=3$, $f(u_5)=5$, $f(u_6)=2$, and $f(u_7)=6$. Furthermore, let $A_1=\{u_1,u_2,u_6\}$, $A_2=\{u_3\}$, $A_3=\{u_4\}$, $A_4=\{u_5\}$, and $A_5=\{u_7\}$. It is straightforward to verify that $f$ is a packing coloring and that $f|_{A_i}$ is a dominating broadcast of $P_7$ for every $i\in[5]$. Therefore, $\chi_{\rho,5}(P_7)\leq6$, and consequently $\chi_{\rho,5}(P_7)=6$.

Finally, let $n\geq8$. By Lemma~\ref{lem:paths_lower}, we have $\chi_{\rho,5}(P_n)\geq5$. We prove that $\chi_{\rho,5}(P_n)\leq5$ by constructing a packing $5$-domatic coloring $f$ of $P_n$ using colors from $[5]$. 
Let $n=8\ell+k$, where $\ell \in \mathbb{N}$ and $k \in [7]_0$. We distinguish seven cases according to the value of $k$. The color pattern $C=[2,1,3,5,4,2,3,1]$ will be used in each of the mentioned cases.\\ 
\textbf{Case 1.} $k \in \{0,1\}$. \\
Color the consecutive vertices of $P_n$ by repeating the pattern $C$ (and if $k=1$, color vertex $u_n$ with color $2$). Define
$A_1=\{u_i\colon i\equiv1,6\pmod 8\}$,
$A_2=\{u_i\colon i\equiv5\pmod 8\}$,
$A_3=\{u_i\colon i\equiv4\pmod 8\}$,
$A_4=\{u_i\colon i\equiv2,7\pmod 8\}$, and
$A_5=\{u_i\colon i\equiv0,3\pmod 8\}$.\\
\textbf{Case 2.} $k=2$. \\
Note that $n=8\ell+2$, $\ell \in \mathbb{N}$. First, if $\ell\geq2$, color the vertices $u_1,u_2,\ldots,u_{8(\ell-1)}$ by repeating the pattern $C$. Then, color (the remaining) ten vertices, namely $u_{8(\ell-1)+1}, \ldots u_n$, by colors $2,1,3,4,2,5,3,2,4,1$, respectively. Further, let
$A_1=\{u_i\colon i\equiv1,6\pmod 8  ~\operatorname{and}~  i \leq 8(\ell-1) \} \cup \{u_{8\ell-6}, u_{8\ell-1}\}$,
$A_2=\{u_i\colon i\equiv5\pmod 8  ~\operatorname{and}~  i \leq 8(\ell-1)\} \cup \{u_{8\ell-4}, u_{8\ell+2}\}$,
$A_3=\{u_i\colon i\equiv4\pmod 8  ~\operatorname{and}~  i \leq 8(\ell-1)\} \cup \{u_{8\ell-2}\}$,
$A_4=\{u_i\colon i\equiv2,7\pmod 8  ~\operatorname{and}~  i \leq 8(\ell-1)\} \cup \{u_{8\ell-5}, u_{8\ell+1}\}$,
$A_5=\{u_i\colon i\equiv0,3\pmod 8  ~\operatorname{and}~ i \leq 8(\ell-1)\} \cup \{u_{8\ell-7}, u_{8\ell-3}, u_{8\ell}\}$.\\ 
\textbf{Case 3.} $k=3$. \\
First, if $\ell\geq2$, let the vertices $u_1, u_2, \ldots, u_{8(\ell-1)}$ receive colors according the pattern $C$. Next, color (the remaining) eleven vertices, $u_{8(\ell-1)+1}, \ldots u_n$, by colors $2,1,4,3,1,5,1,4,2,1,3$, respectively. Finally, define 
$A_1=\{u_i\colon i\equiv1,6\pmod 8  ~\operatorname{and}~  i \leq 8(\ell-1) \} \cup \{u_{8\ell-7}, u_{8\ell-3}, u_{8\ell-1}, u_{8\ell+2}\}$,
$A_2=\{u_i\colon i\equiv5\pmod 8  ~\operatorname{and}~  i \leq 8(\ell-1)\} \cup \{u_{8\ell-2}\}$,
$A_3=\{u_i\colon i\equiv4\pmod 8  ~\operatorname{and}~  i \leq 8(\ell-1)\} \cup \{u_{8\ell-4}, u_{8\ell+3}\}$,
$A_4=\{u_i\colon i\equiv2,7\pmod 8  ~\operatorname{and}~  i \leq 8(\ell-1)\} \cup \{u_{8\ell-6}, u_{8\ell}\}$,
$A_5=\{u_i\colon i\equiv0,3\pmod 8  ~\operatorname{and}~  i \leq 8(\ell-1)\} \cup \{u_{8\ell-5}, u_{8\ell+1}\}$.\\ 
\textbf{Case 4.} $k=4$. \\
If $\ell \geq 2$, color the vertices $u_1, u_2, \ldots, u_{8(\ell-1)}$ by repeating the pattern $C$. Next, color (the remaining) twelve vertices by colors $2,1,4,3,1,5,2,4,3,2,1,5$, respectively. Define 
$A_1=\{u_i\colon i\equiv1,6\pmod 8  ~\operatorname{and}~  i \\
\leq 8(\ell-1) \} \cup \{u_{8\ell-5}, u_{8\ell+2}\}$,
$A_2=\{u_i\colon i\equiv5\pmod 8  ~\operatorname{and}~  i \leq 8(\ell-1)\} \cup \{u_{8\ell-7}, u_{8\ell}\}$,
$A_3=\{u_i\colon i\equiv4\pmod 8  ~\operatorname{and}~  i \leq 8(\ell-1)\} \cup \{u_{8\ell-6}, u_{8\ell-3}, u_{8\ell-1}, u_{8\ell+4}\}$,
$A_4=\{u_i\colon i\equiv2,7\pmod 8  ~\operatorname{and}~ i \leq 8(\ell-1)\} \cup \{u_{8\ell-4}, u_{8\ell+1}\}$,
$A_5=\{u_i\colon i\equiv0,3\pmod 8  ~\operatorname{and}~  i \leq 8(\ell-1)\} \cup \{u_{8\ell-2}, u_{8\ell+3}\}$.\\ 
\textbf{Case 5.} $k=5$. \\
Note that in this case, $n=8\ell+5$, $\ell \in \mathbb{N}$. First, if $\ell \geq 2$, color the first $8(\ell-1)$ vertices of $P_n$ by repeating the pattern $C$. Next, color (the remaining) thirteen vertices $u_{8\ell-7},\ldots,u_{8\ell+5}$ by colors
$2,1,5,3,4,2,1,3,5,4,2,1,3$, respectively.
Define
$A_1=\{u_i:~i\equiv1,6\pmod 8  ~\operatorname{and}~  i\leq8(\ell-1)\}\cup\{u_{8\ell-5},u_{8\ell+3}\}$,
$A_2=\{u_i:~i\equiv5\pmod 8  ~\operatorname{and}~  i\leq8(\ell-1)\}\cup\{u_{8\ell-7},u_{8\ell+1}\}$,
$A_3=\{u_i:~i\equiv4\pmod 8  ~\operatorname{and}~  i\leq8(\ell-1)\}\cup\{u_{8\ell-3},u_{8\ell+5}\}$,
$A_4=\{u_i:~i\equiv2,7\pmod 8  ~\operatorname{and}~  i\leq8(\ell-1)\}\cup\{u_{8\ell-4},u_{8\ell},u_{8\ell+4}\}$, and
$A_5=\{u_i:~i\equiv0,3\pmod 8 ~\operatorname{and}~ i\leq8(\ell-1)\}\cup\{u_{8\ell-6},u_{8\ell-2},u_{8\ell-1},u_{8\ell+2}\}$.\\
\textbf{Case 6.} $k=6$. \\
In this case, first, color the first $8\ell$ vertices of $P_n$ by repeating the pattern $C$. Next, color the remaining six vertices $u_{8\ell+1},\ldots,u_{8\ell+6}$ by colors $2,4,5,2,1,3$, respectively. Define
$A_1=\{u_i:~i\equiv1,6\pmod 8 ~\operatorname{and}~  i\leq8\ell\}\cup\{u_{8\ell+2}\}$,
$A_2=\{u_i:~i\equiv5\pmod 8 ~\operatorname{and}~  i\leq8\ell\}\cup\{u_{8\ell+1},u_{8\ell+5}\}$,
$A_3=\{u_i:~i\equiv4\pmod 8 ~\operatorname{and}~  i\leq8\ell\}\cup\{u_{8\ell+4}\}$,
$A_4=\{u_i:~i\equiv2,7\pmod 8 ~\operatorname{and}~  i\leq8\ell\}\cup\{u_{8\ell+6}\}$, and
$A_5=\{u_i:~i\equiv0,3\pmod 8 ~\operatorname{and}~  i\leq8\ell\}\cup\{u_{8\ell+3}\}$.\\
\textbf{Case 7.} $k=7$. \\
Now $n=8\ell+7$, $\ell \in \mathbb{N}$. First, color the first $8\ell$ vertices of $P_n$ by repeating the pattern $C$. Next, color the remaining seven vertices $u_{8\ell+1},\ldots,u_{8\ell+7}$ by colors $2,5,4,1,3,2,1$, respectively. Define
$A_1=\{u_i :~i\equiv1,6\pmod 8 ~\operatorname{and}~  i\leq8\ell\}\cup\{u_{8\ell+1},u_{8\ell+6}\}$,
$A_2=\{u_i:~i\equiv5\pmod 8 ~\operatorname{and}~ i\leq8\ell\}\cup\{u_{8\ell+2}\}$,
$A_3=\{u_i:~i\equiv4\pmod 8 ~\operatorname{and}~  i\leq8\ell\}\cup\{u_{8\ell+5}\}$,
$A_4=\{u_i:~i\equiv2,7\pmod 8 ~\operatorname{and}~  i\leq8\ell\}\cup\{u_{8\ell+4},u_{8\ell+7}\}$, and
$A_5=\{u_i:~i\equiv0,3\pmod 8 ~\operatorname{and}~  i\leq8\ell\}\cup\{u_{8\ell+3}\}$.

In each of the seven cases, it is straightforward to verify that $f$ is a packing coloring and that $f|_{A_i}$ is a dominating broadcast of $P_n$ for every $i\in[5]$. Consequently, $\chi_{\rho,5}(P_n)\leq5$ for every $n\geq8$. Together with the lower bound $\chi_{\rho,5}(P_n)\geq5$, this yields $\chi_{\rho,5}(P_n)=5$ for every $n\geq8$.
\end{proof}

By the above propositions, when $3\le k\le 5$ we have $\chi_{\rho,k}(P_n)=k$ as soon as $n$ is large enough (for instance, when $k=5$, the equality holds as soon as $n\ge 8$). In the next result, we see that this behavior stops when $k$ is bigger, notably at least when $k\ge 12$. 

\begin{theorem}
    \label{theorem_chi_rho_k_12}
    For any $k \geq 12$ and any $n \geq k$, $\chi_{\rho,k}(P_{n})>k$.
\end{theorem}

\begin{proof}
Let $k\geq12$ and $n\geq k$ be integers. By Lemma~\ref{lem:paths_lower}, we have $\chi_{\rho,k}(P_n)\geq k$.
Suppose that $\chi_{\rho,k}(P_n)=k$, and let $f$ be a packing $k$-domatic coloring using colors from $[k]$. Furthermore, let $\Phi=\{A_1,\ldots,A_k\}$ be a partition of $V(P_n)$ such that $f|_{A_j}$ is a dominating broadcast of $P_n$ for every $j\in[k]$.

We first claim that for each $i\in[k]$, at most one vertex of color $i$ can $f$-dominate $u_1$. Suppose to the contrary that two distinct vertices, $u_j$ and $u_{j'}$, of color $i$ both $f$-dominate $u_1$. Then, assuming $j<j'$, we have $d(u_{j'},u_j)\le d(u_{j'},u_1)\le i$, which is a contradiction with $f$ being a packing coloring. 
This proves the claim.

Next, since $f|{A_j}$ is a dominating broadcast for every $j\in[k]$, each set $A_j$ contains at least one vertex that $f$-dominates $u_1$. As the sets in $\Phi$ are pairwise disjoint, at least $k$ distinct vertices $f$-dominate $u_1$. On the other hand, at most one vertex of each of the $k$ available colors can $f$-dominate $u_1$. Consequently, exactly $k$ vertices $f$-dominate $u_1$, exactly one of each color in $[k]$. For each $i\in[k]$, let $u_{a_i}$ denote the unique vertex of color $i$ that $f$-dominates $u_1$. Note that $d(u_1,u_{a_i})=a_i-1\le i$ and hence, $a_i\le i+1$ for any $i \in [k]$. 

Now, if $n=k$, then $\{u_{a_1},\ldots,u_{a_k}\}=V(P_n)$. In particular, $u_{a_1}$ has color $1$ and satisfies $a_1\leq2$. Hence, $d(u_{a_1},u_4)=4-a_1\geq2>1$, so $u_{a_1}$ does not $f$-dominate $u_4$. Consequently, at most $k-1$ vertices $f$-dominate $u_4$. However, each of the $k$ pairwise disjoint sets in $\Phi$ contains a vertex that $f$-dominates $u_4$, so at least $k$ distinct vertices must $f$-dominate $u_4$, a contradiction. We may therefore assume that $n\geq k+1$.

Note that $\{u_{a_1},\ldots,u_{a_k}\} \subseteq \{u_1,\ldots,u_{k+1}\}$, hence that there is exactly one vertex in $\{u_1,\ldots,u_{k+1}\}$ that does not $f$-dominate $u_1$. Denote this vertex by $u_H$. Since $\{u_{a_1},\ldots,u_{a_k}\}$ contains one vertex of each color in $[k]$, there exists some $i\in[k]$ such that $f(u_H)=f(u_{a_i})$.

We now consider the vertex $u_6$. Since $f|{A_j}$ is a dominating broadcast for every $j\in[k]$, each set $A_j$ contains at least one vertex that $f$-dominates $u_6$. As the sets in $\Phi$ are pairwise disjoint, at least $k$ distinct vertices $f$-dominate $u_6$. Since $a_1\leq2$, we have $d(u_{a_1},u_6)=6-a_1\geq4>1$, and hence $u_{a_1}$ does not $f$-dominate $u_6$. Similarly, since $a_2\leq3$, we have $d(u_{a_2},u_6)=6-a_2\geq3>2$, so $u_{a_2}$ does not $f$-dominate $u_6$ either. Consequently, at most $k-2$ vertices in $\{u_{a_1},\ldots,u_{a_k}\}$ can $f$-dominate $u_6$. Therefore, at least two vertices outside this set must $f$-dominate $u_6$.
Among the vertices in $\{u_1,\ldots,u_{k+1}\}$, the only vertex outside $\{u_{a_1},\ldots,u_{a_k}\}$ is $u_H$. Hence, there exists at least one additional vertex $u_p$ with $p\geq k+2$ that $f$-dominates $u_6$. If $n=k+1$, no such vertex exists, a contradiction. We may therefore assume that $n\geq k+2$.

Let $q=f(u_p)$, where $p\geq k+2$ and $u_p$ $f$-dominates $u_6$. Since $u_p$ $f$-dominates $u_6$, we have $p-6=d(u_p,u_6)\leq q$, or equivalently, $p-q\leq6$. 
In addition, since $p\geq k+2$ and $k\geq12$, we obtain 
\begin{equation}
\label{eq:k12}
q\geq p-6\geq k-4\geq8.
\end{equation}
Recall that $p\geq k+2$ and $q\leq k$. Hence, $d(u_p,u_1)=p-1\geq k+1>q$, so $u_p$ does not $f$-dominate $u_1$. On the other hand, the vertex $u_{a_q}\in \{u_1,\ldots,u_{k+1}\}$ has color $q$ and $f$-dominates $u_1$. Since $f(u_p)=f(u_{a_q})=q$ and $p>a_q$, the packing condition implies that $p-a_q=d(u_p,u_{a_q})>q$. Therefore, $a_q<p-q\leq6$, and hence $a_q\leq5$. We infer that exactly one of the vertices from $N_4[u_1]=\{u_1, u_2, u_3, u_4, u_5\}$ receives the color $q \geq 8$ by $f$.
Further, recall that $a_i\leq i+1$ for every $i\in[k]$. In particular, $a_i\leq 5$ for every $i\in\{1,2,3,4\}$. Hence, $\{u_{a_1},u_{a_2},u_{a_3},u_{a_4}, u_{a_q}\}= N_4[u_1]$. In other words, the vertices in $N_4[u_1]$ are colored with pairwise distinct colors from $\{1,2,3,4,q\}$.

Next, we claim that $u_p$ is the unique vertex in $\{u_{k+2},\ldots, u_n\}$ that $f$-dominates $u_6$. Clearly, any vertex in this set that $f$-dominates $u_6$ must be at distance at most $k$ from $u_6$ and therefore belongs to $\{u_{k+2},\ldots, u_{\min\{n,k+6\}}\}$. Suppose, to the contrary, that another vertex $u_r$ in this set $f$-dominates $u_6$, and let $s=f(u_r)$. As above, we obtain $s\geq8$ and $a_s\leq5$. This means that to a vertex in $N_4[u_1]$ is assigned a color $s$ by $f$. Since the colors assigned to the vertices in $N_4[u_1]$ are precisely $1,2,3,4$, and $q$, and we know that $s \geq 8$, it follows that $s=q$. Hence, $f(u_p)=f(u_r)=q$. Further, since both $u_p$ and $u_r$ belong to $\{u_{k+2},\ldots,u_{\min\{n,k+6\}}\}$, $d(u_p,u_r)\leq4<q$, contradicting the fact that $f$ is a packing coloring. Therefore, $u_p$ is the unique vertex in $\{u_{k+2},\ldots, u_n\}$ that $f$-dominates $u_6$. Note that this implies that $u_H$ must $f$-dominate $u_6$, since at least two vertices outside $\{u_{a_1},\ldots,u_{a_k}\}$ $f$-dominate $u_6$. 


 Now, we consider the position of $u_H$. Recall that the vertices in $N_4[u_1]$ are colored with pairwise distinct colors from $\{1,2,3,4,q\}$ and each of them $f$-dominates $u_1$. Further, since $a_5\leq6$, we have $a_5=6$ which means that $f(u_6)=5$. Similarly, $a_6\leq7$ implies $f(u_7)=6$. Continuing with this argument, we obtain $f(u_{i+1})=i$ for every $i\in\{5,\ldots,q-1\}$. Note that for each $i \in [q]$, $u_{i}$ $f$-dominates $u_1$. Since $u_H$ does not dominates $u_1$, we know that $H \geq q+1$.

Let $c=f(u_H)$. Since $u_H$ $f$-dominates $u_6$, we have $H-6\leq c$.
Furthermore, there is a unique vertex $u_{a_c}$ of color $c$ that $f$-dominates $u_1$. Since $u_H$ does not $f$-dominate $u_1$ it is clear that $a_c < H$.
The packing condition applied to $u_H$ and $u_{a_c}$ now gives $H-a_c>c$. 
Combining this inequality with $H-6\leq c$, we obtain $H-a_c>H-6$, and therefore $a_c\leq5$. 
Since $\{f(u_1),\ldots,f(u_5)\}=\{1,2,3,4,q\}$, we conclude that $c\in\{1,2,3,4,q\}$.
Now, we show that $c\neq q$. Recall that $H\geq q+1$ and $p\leq q+6$. Moreover, $u_H \in \{u_1,\ldots,u_{k+1}\}$, whereas $u_p \in \{u_{k+2},\ldots,u_{\min\{n,k+6\}}\}$, so $H<p$. It follows that $d(u_H,u_p)=p-H\leq(q+6)-(q+1)=5$. If $c=q$, then $u_H$ and $u_p$ are distinct vertices of color $q$ satisfying $d(u_H,u_p)\leq5$, contradicting the packing condition since $q\geq8$. Therefore, $c\neq q$, and consequently $c \leq 4$. 

We now show that $k\in\{12,13\}$. Since $u_H$ $f$-dominates $u_6$, we have $H-6\leq c=f(u_H)$. Hence, $H\leq c+6\leq10$, since $c\leq4$. Combining this with $H\geq q+1$, we obtain $q+1\leq H\leq10$, and therefore $q\leq9$. Since we have already shown that $q\geq k-4$ (see~\eqref{eq:k12}), it follows that $k-4\leq q\leq9$, and hence $k\leq13$. Together with $k\geq12$, this yields $k\in \{12,13\}$. In continuation of the proof we distinguish two cases. 

First, suppose that $k=13$. Then $q\geq k-4=9$, and since $q\leq9$, we obtain $q=9$. The inequalities $q+1\leq H\leq10$ then imply that $H=10$. Moreover, $H-6\leq c\leq4$ yields $c=4$. Finally, $p\geq k+2=15$ and $p\leq q+6=15$, so $p=15$. Thus, when $k=13$, we necessarily have $p=15$, $q=9$, $H=10$, and $c=4$. In other words, $f(u_{10})=4$ and $f(u_{15})=9$.
We now consider the vertex $u_4$. The vertex $u_{a_1}$ of color $1$ does not $f$-dominate $u_4$, since $a_1\leq2$ and hence $d(u_{a_1},u_4)=4-a_1\geq2>1$. Moreover, neither $u_H=u_{10}$ nor $u_p=u_{15}$ $f$-dominates $u_4$, because $d(u_{10},u_4)=6>4=f(u_{10})$ and $d(u_{15},u_4)=11>9=f(u_{15})$.
No other vertex in $V(P_n)\setminus \{u_{a_1},\ldots,u_{a_{13}}\}$ can $f$-dominate $u_4$. Indeed, every such vertex different from $u_H$ belongs to $\{u_{15},u_{16},\ldots,u_n\}$. If such a vertex $f$-dominated $u_4$, then it would also $f$-dominate the closer vertex $u_6$. However, $u_{15}$ is the unique vertex in this set that $f$-dominates $u_6$, and we have already shown that it does not $f$-dominate $u_4$.
Therefore, only vertices in $\{u_{a_2},\ldots,u_{a_{13}}\}$ can $f$-dominate $u_4$, so at most $12$ vertices $f$-dominate $u_4$. On the other hand, each of the $13$ pairwise disjoint sets in $\Phi$ contains a vertex that $f$-dominates $u_4$. Thus, at least $13$ distinct vertices must $f$-dominate $u_4$, a contradiction.

Now suppose that $k=12$. Recall that $q+1\leq H\leq10$ and $8\leq q\leq9$. If $q=8$, then $H\in \{9,10\}$, whereas if $q=9$, then $H=10$. Moreover, since $H-6\leq c\leq4$, $H=9$ implies that $c\in \{3,4\}$, while $H=10$ implies that $c=4$.
We also have $p\geq k+2=14$ and $p\leq q+6$. If $q=8$, then $p=14$, while if $q=9$, then $p\in \{14,15\}$. Hence, the possible pairs $(p,q)$ are $(14,8)$, $(14,9)$, and $(15,9)$, while the possible pairs $(H,c)$ are $(9,3)$, $(9,4)$, and $(10,4)$.
In each case, $u_H$ does not $f$-dominate $u_4$, since $d(u_H,u_4)=H-4\geq5>c$. The vertex $u_p$ does not $f$-dominate $u_4$ either. Indeed, if $(p,q)=(14,8)$, then $d(u_p,u_4)=10>8=q$; if $(p,q)=(14,9)$, then $d(u_p,u_4)=10>9=q$; and if $(p,q)=(15,9)$, then $d(u_p,u_4)=11>9=q$.
Furthermore, no other vertex in $V(P_n)\setminus \{u_{a_1},\ldots,u_{a_{12}}\}$ can $f$-dominate $u_4$. Indeed, every such vertex different from $u_H$ lies to the right of $u_{13}$. If it $f$-dominated $u_4$, then it would also $f$-dominate the closer vertex $u_6$. By the uniqueness of $u_p$, this vertex would have to be $u_p$, but we have already shown that $u_p$ does not $f$-dominate $u_4$.
Finally, the vertex $u_{a_1}$ of color $1$ does not $f$-dominate $u_4$, since $a_1\leq2$ and therefore $d(u_{a_1},u_4)=4-a_1\geq2>1$. Hence, at most $11$ vertices $f$-dominate $u_4$, which is a contradiction since each of the $12$ pairwise disjoint sets in $\Phi$ must contain a vertex that $f$-dominates $u_4$.
\end{proof}


\subsection{The infinite path}

Note that Lemma~\ref{lem:paths_lower} does not apply to the two-way infinite path $P_\infty$, thus $\chi_{\rho,k}(P_\infty)$ could be smaller than $k$. In fact, this happens for all $k\in\{4,5,6\}$ as the following result shows.

\begin{theorem}
    \label{infinite_paths}
For the path $P_\infty$, 
$$\chi_{\rho, k}(P_\infty) =
\left\{\begin{array}{ll}
3 & \text{if } k\in \{3,4\}\,,\\
4 & \text{if } k=5\,.
\end{array}\right.$$
In addition, $\chi_{\rho, 6}(P_\infty) \in \{4, 5\}$. 
\end{theorem}
\begin{proof}
    Let $P_\infty$ be the two-way infinite path with $V(P_\infty)=\{\ldots, u_{-3}, u_{-2}, u_{-1}, u_0, u_1, u_2, u_3, \ldots\}$.

    First, let $k=3$. Since $P_4$ is a subgraph of $P_\infty$ and $\chi_\rho(P_4)=3$, it follows that $\chi_{\rho, 3}(P_\infty) \geq 3$. \\
    Next, define a packing coloring $f$ of $P_\infty$ by assigning colors to consecutive vertices of $P_\infty$ according to the repeating pattern $[1,2,1,3]$. Then $\Pi=\{D_1, D_2, D_3\}$, where $D_1=\{u \in V(P_\infty) : f(u)=1\}$, $D_2=\{u \in V(P_\infty) : f(u)=2\}$ and $D_3=\{u \in V(P_\infty) : f(u)=3\}$, is a partition of $V(P_\infty)$ such that for each $i \in [3]$, $f|_{D_i}$ is a dominating broadcast. Hence, $\chi_{\rho, 3}(P_\infty) = 3$.

    Next, consider the case when $k=4$. Using the fact that $\chi_{\rho, 3}(P_\infty) = 3$ and Observation~\ref{obs:monotonicity}, we derive that $\chi_{\rho, 4}(P_\infty) \geq 3$. 
    Applying the periodic pattern $[1, 2, 1, 3]$ to the consecutive vertices of $P_\infty$ such that $f(u_0)=2$, we obtain a packing $4$-domatic coloring $f$ of $P_\infty$. Indeed, $f|_{D_1}$, where $D_1=\{u \in V(P_\infty) :~f(u)=1\}$, is a dominating broadcast. Further, let $D_2=\{u_j :~j \equiv 0 \pmod 6\}$, $D_3=\{u_j;~j \equiv 2 \pmod 6\}$ and $D_4=\{u_j :~j \equiv 4 \pmod 6)\}$. Note that for any $t \in \mathbb{Z}$, $u_{6t}, u_{6(t+1)} \in D_2$, $\{f(u_{6t}), f(u_{6(t+1)}) \}=\{2,3\}$ and $d(u_{6t}, u_{6(t+1)})=6$. Hence, each vertex of $P_\infty$ is $f$-dominated by a vertex of $D_2$, implying that $f|_{D_2}$ is a dominating broadcast. By an analogous argument, $f|_{D_3}$ and $f|_{D_4}$ are also dominating broadcasts. Therefore, $f$ is a packing $4$-domatic coloring of $P_\infty$ and consequently, $\chi_{\rho, 4}(P_\infty)=3$. 

    Now, let $k=5$. Suppose that there exists a packing $5$-domatic coloring $f$ of $P_\infty$ using only $3$ colors. Then $f$ must assign colors to the consecutive vertices of $P_\infty$ according to the periodic pattern $[1,2,1,3]$. Without loss of generality we may assume that $f(u_0)=1$ and $f(u_1)=2$. Consequently, $f(u_{-3})=2$, $f(u_{-2})=1$, $f(u_{-1})=3$, $f(u_2)=1$, and $f(u_3)=3$. We observe that $u_0$ can be $f$-dominated only by the vertices $u_{-1}, u_0, u_1$ and $u_3$. Hence, it is impossible to partition $V(P_\infty)$ into five sets, each inducing a dominating broadcast. Therefore, $\chi_{\rho, 5}(P_\infty) > 3$. 
    
    Applying the periodic pattern $[4,1,3,2,1]$ to the consecutive vertices of $P_\infty$ so that $f(u_0)=4$, we get a packing $5$-domatic coloring $f$ of $P_\infty$. Indeed, let $D_1=\{u \in V(P_\infty) :\, f(u)=1\}$. Then, $f|_{D_1}$ is a dominating broadcast. 
    Next, let $D_2=\{u_j :\,j \equiv 0, 7, 13 \pmod {20}\}$, $D_3=\{u_j :\, j \equiv 2,8,15 \pmod {20}\}$, $D_4=\{u_j:\, j \equiv 3,10,17 \pmod {20}\}$ and $D_5=\{u_j :\, j \equiv 5, 12, 18 \pmod {20}\}$. We observe for any $t \in \mathbb{Z}$, $f(u_{20t})=4$, $f(u_{20t+7})=3$, $f(u_{20t+13})=2$, $f(u_{20t+20})=4$ and $d(u_{20t}, u_{20t+7})=7$, $d(u_{20t+7}, u_{20t+13})=6$ and $d(u_{20t+13}, u_{20t+20})=7$. It follows that $f|_{D_2}$ is a dominating broadcast. By analogous arguments, we derive that $f|_{D_3}$, $f|_{D_4}$ and $f|_{D_5}$ are also dominating broadcasts. Therefore, $f$ is a packing $5$-domatic coloring of $P_\infty$ using $4$ colors, thus, $\chi_{\rho, 5}(P_\infty)=4$. 

    Finally, consider the case when $k=6$. Using Observation~\ref{obs:monotonicity} and the fact that $\chi_{\rho, 5}(P_\infty)=4$, we derive that $\chi_{\rho, 6}(P_\infty) \geq 4$. 
    In order to prove that $\chi_{\rho, 6}(P_\infty) \leq 5$, we form a packing coloring $f$ of $P_\infty$ using $5$ colors by applying the periodic pattern $[5,1,2,4,1,3]$ to the consecutive vertices of $P_\infty$ such that $f(u_0)=5$. We claim that the presented coloring is a packing $6$-domatic coloring of $P_\infty$. Clearly, for $D_1=\{u \in V(P_\infty) : f(u)=1\}$, $f|_{D_1}$ is a dominating broadcast. Further, let $D_2=\{u_j : j \equiv 0,8,15,23 \pmod {30}\}$, $D_3=\{u_j : j \equiv 2,9,17,24 \pmod {30}\}$, $D_4=\{u_j : j \equiv 3,11,18,26 \pmod {30}\}$, $D_5=\{u_j : j \equiv 5,12,20,27 \pmod {30}\}$ and $D_6=\{u_j : j \equiv 6,14,21,29 \pmod {30}\}$. We observe that for any $i \in \{2,3,4,5,6\}$, the consecutive vertices of $D_i$ are colored according to the pattern $[5,2,4,3]$. Moreover, for any $i \in \{2,3,4,5,6\}$, any two consecutive vertices of $D_i$ colored with $5$ and $2$, or with $4$ and $3$, are at distance $8$, while any two consecutive vertices colored with $2$ and $4$, or with $3$ and $5$, are at distance $7$. Hence, $f|_{D_i}$ is a dominating broadcast for any $i \in \{2,3,4,5,6\}$, which implies that $\chi_{\rho, 6}(P_\infty) \leq 5$. Therefore, $\chi_{\rho, 6}(P_\infty) \in \{4, 5\}$.
\end{proof}

Note that the above proposition determines the exact values of $\chi_{\rho,k}(P_\infty)$ for $k\in\{3,4,5\}$. 
For $k=6$, the problem is reduced to determining which of the two possible values in $\{4,5\}$ is attained by $\chi_{\rho,6}(P_\infty)$,  a question we leave for the reader. The cases $k\geq 7$ remain open and are left for future investigation.

Lower bounds on $\chi_{\rho, k}(P_\infty)$ seem to be more challenging. We are able to obtain the following lower bound.
\begin{proposition}
For any integer $k\ge 3$,
    $\chi_{\rho, k}(P_\infty)> \frac{k-1}{2}$. 
\end{proposition}
\begin{proof} Suppose to the contrary that $\chi_{\rho, k}(P_\infty) \leq  \frac{k-1}{2}$. This means that there exists a packing $k$-domatic coloring using colors from $\{1,2,\ldots,\lfloor \frac{k-1}{2} \rfloor\}$. 
Denote such a coloring by $f$ and let $\Pi=\{A_1, \ldots, A_k\}$ be an arbitrary partition of $V(P_\infty)$ such that $f|_{A_i}$ is a dominating broadcast for every $i \in [k]$. 
Consider the vertex $u_0$. Since $f$ uses colors from $\{1,2,\ldots,\lfloor \frac{k-1}{2} \rfloor\}$, only the vertices in $X=\{u_{-\lfloor \frac{k-1}{2} \rfloor}, u_{-\lfloor \frac{k-1}{2} \rfloor +1},\ldots, u_{\lfloor \frac{k-1}{2} \rfloor}\}$ can $f$-dominate $u_0$. 
If $k$ is even, then $|X|=k-1<k$, a contradiction to $f$ being a packing $k$-domatic coloring.
Next, suppose that $k$ is odd. In this case, $|X|=k$ and thus every vertex in $X$ must $f$-dominate $u_0$. 
Therefore, $u_{-\lfloor \frac{k-1}{2} \rfloor} =u_{ -\frac{k-1}{2}}$ $f$-dominates $u_0$. Since $d(u_0, u_{ -\frac{k-1}{2}})=\frac{k-1}{2}$, we have $f(u_{-\frac{k-1}{2}}) \geq \frac{k-1}{2}$. By our assumption, $f$ uses only colors from $\{1,2,\ldots, \frac{k-1}{2}\}$, hence $f(u_{-\frac{k-1}{2}})= \frac{k-1}{2}$.
Similarly, since $u_{\lfloor \frac{k-1}{2} \rfloor} =u_{ \frac{k-1}{2}}$ $f$-dominates $u_0$ and $d(u_0, u_{\frac{k-1}{2}})=\frac{k-1}{2}$, we know that $f(u_{\frac{k-1}{2}})=\frac{k-1}{2}$.
Further, the facts that $u_{ -\frac{k-1}{2}+1}$, $u_{\frac{k-1}{2}-1}$ $f$-dominate $u_0$ and $f$ is a packing coloring imply that $f(u_{ -\frac{k-1}{2}+1})=f(u_{\frac{k-1}{2}-1})=\frac{k-1}{2}-1$. Applying the same arguments successively to the vertices $u_{ -\frac{k-1}{2}+2}$, $u_{\frac{k-1}{2}-2}$, $u_{ -\frac{k-1}{2}+3}$ and $u_{\frac{k-1}{2}-3}, \ldots, u_{-1}$ and $u_1$ we derive that $f(u_{ -\frac{k-1}{2}+j})=f(u_{\frac{k-1}{2}-j})=\frac{k-1}{2}-j$ for any $j \in \{0,1,2, \ldots, \frac{k-1}{2}-1\}$. But then, $f(u_0)=f(u_{\frac{k-1}{2}-j})=\frac{k-1}{2}-j$ for some $j$, a contradiction to $f$ being a packing coloring since $d(u_0, u_{\frac{k-1}{2}-j})=\frac{k-1}{2}-j$. 
 \end{proof}

\bigskip

\subsection{Cycles}
It is well-known that for a cycle $C_n$ on $n$ vertices, $\chi_\rho(C_n)=3$ if $n=3$ or $n\equiv 0\pmod{4}$, and $\chi_\rho(C_n)=4$ otherwise. 
As an immediate consequence, $\chi_{\rho,k}(C_n)\ge \chi_\rho(C_n)=3$ whenever $n=3$ or $n\equiv 0\pmod{4}$, and $\chi_{\rho,k}(C_n)\ge \chi_\rho(C_n)\ge 4$ for all $n \geq 4$ with the property that $n\not\equiv 0\pmod{4}$. When $k=3$ these lower bounds are actually exact values. Note that a packing $3$-domatic coloring of $C_3$ can be obtained by assigning distinct colors from $[3]$ to its three vertices. Thus, $\chi_{\rho,3}(C_3)=3$.
\begin{proposition}
	\label{3_cycles}
	For any cycle $C_n$ with $n \geq 4$, 
	$$\chi_{\rho, 3}(C_n) =
	\left\{\begin{array}{ll}
		4 & \text{if } n \not\equiv 0 \ (\mathrm{mod} \ 4)\,,\\
		3 & \text{if } n \equiv 0 \ (\mathrm{mod} \ 4)\,.
	\end{array}\right.$$
\end{proposition}

\begin{proof}
Let $n \geq 4$ be an arbitrary integer. By the above observation, we only need to prove the upper bounds on $\chi_{\rho, 3}(C_n)$. First, if $n \equiv 0 \ (\mathrm{mod} \ 4)$, we form a packing $3$-domatic coloring $f$ of $C_n$ as follows: color the consecutive vertices of $C_n$ successively by repeating the color pattern $[1, 2, 1, 3]$. 
If $n \not\equiv 0 \ (\mathrm{mod} \ 4)$, we distinguish three cases according  to the value of $n$. \newline
\textbf{Case 1. $n \equiv 1 \ (\mathrm{mod} \ 4)$.} \newline
First, let $f$ assign color $4$  to the vertex $u_1$. Next, apply the color pattern $[1, 2, 1, 3]$ to color the remaining consecutive vertices of $C_n$. \newline
\textbf{Case 2. $n \equiv 2 \ (\mathrm{mod} \ 4)$.} \newline
Let $f$ assign colors $2$ and $4$ to the vertices $u_1$ and $u_2$, respectively. Next, color the remaining vertices successively using the color pattern $[1,2,1,3]$. \newline	
\textbf{Case 3. $n \equiv 3 \ (\mathrm{mod} \ 4)$.} \newline
Let the vertices $u_1,u_2,u_3$ receive colors $1,2,4$ by $f$, respectively. Next, color the remaining vertices successively using the color pattern: $[1, 2, 1, 3]$.

It is straightforward to verify that, in each of the cases, $f$ is a packing coloring. Next, let $A_1=\{u_j;~f(u_j)=1 \}\cup \{u_k; f(u_k)=4\}$ and  $A_i=\{u_j;~f(u_j)=i\}$ for $2\leq i \leq 3$. Since $\Pi=\{A_1, A_2, A_3\}$ is a partition of $V(P_n)$ and for each $i \in [3]$, $f|_{A_i}$ is a dominating broadcast, we are done. 
\end{proof}

In the remainder of this section, we focus on establishing $\chi_{\rho, 4}(C_n)$ for all $n\ge 4$. Interestingly, the cycles $C_8$ and $C_{20}$ behave a bit awkwardly, and we deal with them separately.  

\begin{proposition}
\label{4_cycle}
$\chi_{\rho, 4}(C_8) =4$.
\end{proposition}

\begin{proof}
 Let  $C_8=u_1u_2\cdots u_8u_1$. First, we prove that $\chi_{\rho, 4}(C_8)\geq4$. Suppose to the contrary that $\chi_{\rho,4}(C_8)=3$. Then there exists a packing $4$-domatic coloring $f$ of $C_8$ using three colors. Let $\Pi=\{A_1,A_2,A_3,A_4\}$ be a partition of $V(C_8)$ such that, for each $i\in[4]$, $f|_{A_i}$ is a dominating broadcast.
 Since $f$ is a packing coloring using $3$ colors, among any two consecutive vertices of $C_8$ one receives color $1$ by $f$. Indeed, suppose that $u_3$ has color $2$ and $u_4$ has color $3$. Then neither $u_1$ nor $u_2$ can be colored by a color in $\{2,3\}$, and only one of them can receive color $1$ by $f$, a contradiction. Therefore, $f(u_{2i-1})=1$ for any $i \in [4]$ or $f(u_{2i})=1$ for any $i \in [4]$. Without loss of generality we may assume that $f(u_2)=f(u_4)=f(u_6)=f(u_8)=1$, which implies the following: $f(u_1)=f(u_5)=3$, $f(u_3)=f(u_7)=2$ or $f(u_1)=f(u_5)=2$, $f(u_3)=f(u_7)=3$. Due to the symmetry, we can assume that $f(u_1)=f(u_5)=3$, $f(u_3)=f(u_7)=2$. 
 Since $|V(C_8)|=4$ and $|\Pi|=4$, we infer that $|A_i|=2$ for each $i \in [4]$ as $d(u_i,u_{i+4})=4$. However, if $u_{2i}\in A_j$ for some $i\in [4]$, then $f|_{A_j}$ can be a dominating broadcast only if $u_1$ or $u_5$ belongs to $A_j$. Consequently, only two of the sets $A_j$ can be dominating broadcasts, a contradiction. Hence $\chi_{\rho, 4}(C_8) \geq4$.

 To prove that $\chi_{\rho, 4}(C_8) \leq4$, we form a packing $4$-domatic coloring $f$ of $C_8$ using $4$ colors. Let $f$ assign colors $2,1,3,4,1,2,3,1$ successively to the vertices $u_1,u_2,\ldots,u_8$. Clearly, $f$ is a packing coloring. Let $A_i=\{u_j:\, f(u_j)=i\}$ for each $i\in[4]$. Since $\Pi=\{A_1,A_2,A_3, A_4\}$ is a partition of $V(C_8)$, and  $f|_{A_i}$ is a dominating broadcast for each $i\in[4]$, we conclude that $\chi_{\rho, 4}(C_8) = 4$. 
\end{proof}

\begin{proposition}
	\label{pro_20_cycle}
$\chi_{\rho, 4}(C_{20}) =4$.
\end{proposition}

\begin{proof}
First, we prove that $\chi_{\rho, 4}(C_{20})\geq4$. Suppose to the contrary that $\chi_{\rho,4}(C_{20})=3$. Then there exists a packing $4$-domatic coloring $f$ of $C_{20}$ using three colors. Let $\Pi=\{A_1,A_2,A_3,A_4\}$ be a partition of $V(C_{20})$ such that, for each $i\in[4]$, $f|_{A_i}$ is a dominating broadcast.
Let  $C_{20}=u_0u_1u_2\cdots u_{19}u_0$. As in the preceding proof, we note that since $f$ is a packing coloring using $3$ colors, among any two consecutive vertices of $C_{20}$ one receives color $1$ by $f$, and so the vertices with color $1$ alternate. 
Without loss of generality, we may assume that $f(u_{2j+1})=1$ for each $j\in[9]_0$, while $f(u_{4i})=3$ and $f(u_{4i+2})=2$ for each $i\in[4]_0$. In particular, exactly  five vertices in $V(C_{20})$ are assigned color $t$ for each $t\in\{2,3\}$. 

We claim that for every $i \in [4]$, $|A_i|\geq4$. Suppose to the contrary that exists $i\in[4]$ such that $|A_i|\leq3$. Without loss of generality, let $|A_1|\leq3$. Suppose that there exists a vertex $u_{i_j}\in A_1$ with $f(u_{i_j})\in \{1,2\}$. Then $f|_{A_1}$ is not a dominating broadcast, because $\sum_{j\in [3]}(2f(u_{i_j})+1)<20$ for $A_1=\{u_{i_1},u_{i_2},u_{i_3}\}$, 
a contradiction. 
Hence,  to all vertices in $A_1$ must be assigned color $3$ by $f$. Since there are five vertices in $C_{20}$ with color $3$,  two of these vertices must be at distance four. We can assume that $A_1=\{u_0,u_4,u_t\}$.  However, if $t=12$, then $d(u_8,u_k)> f(u_k)=3$ for each $k\in\{0,4,12\}$. If $t=8$ or $t=16$, then $d(u_{12},u_k)> f(u_k)=3$ for each $k\in\{0,4,t\}$. In either case, we get a contradiction with the assumption that $f|_{A_1}$ is a dominating broadcast. Thus, $|A_i|\geq 4$ for every $i\in[4]$. Relabeling the sets if necessary, we may assume that $|A_1|\geq |A_2|\geq |A_3|\geq |A_4|\geq4$.

To verify that $\chi_{\rho,4}(C_{20}) \geq 4$, we performed an exhaustive computational search over all possible 4-partitions of $C_{20}$. Since $|A_1|+|A_2|+|A_3|+|A_4|=20$, 
the only 
possibilities for $4$-tuples $(|A_1|, |A_2|, |A_3|, |A_4|)$ are $(4,4,4,8)$, $(4,4,5,7)$, $(4,4,6,6)$, $(4,5,5,6)$, and  $(5,5,5,5)$. The program examines each of these five cases and verifies that none admits a valid partition yielding a packing $4$-domatic coloring. See Appendix A for details. Hence, using three colors do not suffice to partition $V(C_{20})$ into $4$ dominating broadcast parts, which implies $\chi_{\rho,4}(C_{20}) \ge 4$.

To prove that $\chi_{\rho,4}(C_{20}) \leq 4$, it suffices to construct a packing  $4$-domatic coloring $f$ of $C_{20}$ using four colors. Define $f$ by repeating the  pattern $[1,3,2,1,4]$ on the  consecutive vertices of $C_{20}$. Then $f$ is a packing coloring. For each $i\in[4]$, let $A_i=\{u_j:\, f(u_j)=i\}$.  Since $\Pi=\{A_1,A_2,A_3,A_4\}$ is a partition of $V(C_{20})$, and $f|_{A_i}$ is a dominating broadcast for each $i\in[4]$, we conclude that $\chi_{\rho,4}(C_{20})\leq4$.
\end{proof}

\begin{proposition}
	\label{prop_chi_rho_4_finite_paths}
	For any cycle $C_n$ with $n \geq 4$,~and~$n\notin\{8,20\}$, 
	$$\chi_{\rho, 4}(C_n) =
	\left\{\begin{array}{ll}
		5 & \text{if } n=4\,,\\
		4 & \text{if } n \not\equiv 0 \ (\mathrm{mod} \ 4)\,,\\
		3 & \text{otherwise}  \,.
	\end{array}\right.$$
\end{proposition}

\begin{proof}
  Let $V(C_4)=\{u_1,u_2,u_3,u_4\}$. Suppose, to the contrary, that $\chi_{\rho,4}(C_4)=4$. This means that there exists a packing $4$-domatic coloring $f$ of $C_4$ using $4$ colors. Let $\Pi=\{A_1, A_2, A_3, A_4\}$ be a partition of $V(C_4)$ such that for each $i \in [4]$, $f|_{A_i}$ is a dominating broadcast. Since $|V(C_4)|=4$, it follows that $|A_i|=1$ for each $i \in [4]$. Additionally, since $f|_{A_i}$ is a dominating broadcast, we conclude that $f(u_i) \neq 1$ for any $i \in [4]$. 
  Further, since diam$(C_4)=2$, $f$ assigns each of the colors to at most one vertex of $C_4$, thus $\chi_{\rho, 4}(C_4) \geq 5$. 
By setting $f(u_1)=5$, $f(u_2)=2$, $f(u_3)=3$ and $f(u_4)=4$, we obtain a packing $4$-domatic coloring $f$ of $C_4$ using $5$ colors, thus $\chi_{\rho, 4}(C_4) =5$.	
	
Next, let $n \geq 5$. If $n \not\equiv 0 \ (\mathrm{mod} \ 4)$, then $\chi_{\rho, 4}(C_n)\geq\chi_{\rho}(C_n)=4$. To prove that $\chi_{\rho, 4}(C_n) \leq 4$, we construct a packing $4$-domatic coloring $f$ of $C_n$ using $4$ colors. Let $C_n=u_1u_2\cdots u_nu_1$. To describe this coloring, we distinguish two cases with respect to the value of $n$.  \vspace{1mm}\\
\textbf{Case 1. $n \equiv 0,1,2 \ (\mathrm{mod} \ 5)$.}\\ 
In this case,  we apply the color pattern $[1, 3, 2, 1, 4]$ to color the  consecutive vertices of $C_n$. For $n \equiv 1 \ (\mathrm{mod} \ 5)$ and  $n \equiv 2 \ (\mathrm{mod} \ 5)$, we adjust the colors of the last one or two vertices while keeping the rest of the coloring unchanged. Specifically,  set $f(u_n)=2$ if  $n \equiv 1 \ (\mathrm{mod} \ 5)$, and  set $f(u_{n-1})=1$ and  $f(u_n)=2$ if $n \equiv 2 \ (\mathrm{mod} \ 5)$. For  $n \equiv 0 \ (\mathrm{mod} \ 5)$, no adjustment is required.
\vspace{1mm}\\
\textbf{Case 2. $n \equiv 3,4 \ (\mathrm{mod} \ 5).$}\\ 
We apply the color pattern $[2,1,3,4,1]$ to color  consecutive vertices of $C_n$. For $n\equiv 3\ (\mathrm{mod} \ 5)$, we keep the pattern for vertices $u_1,u_2,\dots,u_{n-3}$ and set $f(u_{n-2})=2$, $f(u_{n-1})=3$, and $f(u_n)=1$. For $n\equiv 4\ (\mathrm{mod} \ 5)$, we retain the pattern on $u_1,u_2,\dots,u_{n-4}$ and set $f(u_{n-3})=2$, $f(u_{n-2})=1$, $f(u_{n-1})=3$, and $f(u_n)=1$.  

It is straightforward to verify that $f$ is a packing coloring in each case. Next, for each $i \in [4]$, let $A_i=\{u_j : f(u_j)=i\}$. Since $\Pi=\{A_1, A_2, A_3, A_4\}$ is a partition of $V(P_n)$ and $f|_{A_i}$ is a dominating broadcast  for each $i \in [4]$, we conclude that $\chi_{\rho, 4}(C_n) \leq 4$.

Finally, let $n \equiv 0 \ (\mathrm{mod} \ 4)$, where $n\notin\{4,8,20\}$. Thus, $\chi_{\rho, 4}(C_n)\geq\chi_{\rho}(C_n)=3$, and let us prove $\chi_{\rho, 4}(C_n) \leq 3$. Starting from the vertex $u_1$ up to $u_n$, we color the vertices using the pattern $[3, 1, 2, 1]$ by periodically repeating it. We will present a partition $\Pi=\{A_1, A_2, A_3, A_4\}$ of $V(C_n)$ such that  $f|_{A_i}$ is a dominating broadcast  for each $i \in [4]$.  Letting $n=4t$, we describe the mentioned partition by distinguishing five cases according to the value of $n$.   \vspace{1mm}\\
\textbf{Case 1.} $4t \equiv 0 \ (\mathrm{mod} \ 12)$.\\
In this case,  let $A_1=\{u_{2i} : i\in [2t]\}$, $A_2=\{u_{6i+1} : i\in [\frac{2t}{3}-1]_0\}$, $A_3=\{u_{6i+3} : i\in [\frac{2t}{3}-1]_0\}$,  and $A_4=\{u_{6i+5} : i\in [\frac{2t}{3}-1]_0\}$.
\vspace{1mm}\\
\textbf{Case 2.} $4t \equiv 4 \ (\mathrm{mod} \ 12)$ and $t$ is odd.\\
In this case, let $A_1=\{u_{4t-4}, u_{4t-6}, u_{6i+1} : i\in [\lfloor\frac{2t}{3}\rfloor-1]_0\}$, $A_2=\{u_2, u_4, u_{6i+3} : i\in [\lfloor\frac{2t}{3}\rfloor]\}$, $A_3=\{u_3,u_{4t-3},u_{2i} : i\in \{3,\ldots,2t-4\}\}$,  and $A_4=\{u_{4t}, u_{4t-2}, u_{6i+5} : i\in [\frac{2t-5}{3}]_0\}$.
\vspace{1mm}\\
\textbf{Case 3.} $4t \equiv 8 \ (\mathrm{mod} \ 12)$ and $t$ is odd.\\
In this case, let $A_1=\{u_{4t-4}, u_{4t-6},  u_{4t-8}, u_{4t-10}, u_{6i+1} : i\in [\frac{2t-7}{3}]_0\}$, $A_2=\{u_2, u_4, u_{4t}, u_{4t-2}, u_{6i+3} : i\in
[\frac{2t-4}{3}]\}$, $A_3=\{u_3,u_{4t-3}, u_{4t-9}, u_{4t-15}, u_{2i} : i\in \{3,\ldots,2t-10\}\}$,  and $A_4=\{u_{4t-1}, u_{4t-7}, u_{4t-12}, u_{4t-14},\newline
 u_{4t-16}, u_{4t-18}, u_{6i+5} : i\in[\frac{2t-13}{3}]_0\}$.
 \vspace{1mm}\\
\textbf{Case 4.} $4t \equiv 4 \ (\mathrm{mod} \ 12)$ and $t$ is even.\\
In this case, let $A_1=\{u_{4t-4}, u_{4t-6}, u_{6i+1} : i\in [\frac{2t-5}{3}]_0\}$, $A_2=\{u_{4t-3}, u_{2i} : i\in [2t-4]\}$, $A_3=\{u_{6i+3} : i\in [\frac{2t-2}{3}]_0\}$,  and $A_4=\{u_{4t}, u_{4t-2}, u_{6i+5} : i\in [\frac{2t-5}{3}]_0\}$.
\vspace{1mm}\\
\textbf{Case 5.} $4t \equiv 8 \ (\mathrm{mod} \ 12)$ and $t$ is even.\\
Let $A_1=\{u_{4t-4}, u_{4t-6}, u_{4t-8}, u_{4t-10}, u_{6i+1} : i\in [\frac{2t-7}{3}]_0\}$, $A_2=\{u_{4t-3}, u_{4t-9}, u_{4t-15}, u_{2i} : i\in [2t-10]\}$, $A_3=\{u_{4t},u_{4t-2},u_{6i+3} : i\in [\frac{2t-4}{3}]_0\}$,  and $A_4=\{u_{4t-1}, u_{4t-7},  u_{4t-12},  u_{4t-14},  u_{4t-16},  u_{4t-18},
 u_{6i+5} : i\in [\frac{2t-13}{3}]_0\}$.
 
It is straightforward to verify that in each of the cases, $f|_{A_i}$ is a dominating broadcast for every $i\in [4]$. Therefore, $\chi_{\rho, 4}(C_n) \leq 3$, and the proof is complete.
\end{proof}

\section{Packing domatic number in graphs with $\delta\ge 2$}
\label{sec:domatic}

In Observation~\ref{obs:domatic} we proved that $\dchiS(G)$ is at least $2$ in any isolate-free graph $G$, which is a direct consequence of the classical bound $\dom(G)\ge 2$. In the next result, we improve this bound for graphs with minimum degree at least $2$,  which indicates the strength of the new setting.

\begin{theorem}
\label{thm:domlower}
If $G$ is a connected graph with $\delta(G)\ge 2$, then $$\dchir(G)\ge 3.$$
\end{theorem}
\begin{proof}
Note that $\chi_\rho(G)\ge 3$ for any connected graph unless $G$ is a star $K_{1,r}$. Thus, by the assumption $\delta(G)\ge 2$, we infer that $\chi_\rho(G)\ge 3$. Note that for an infinite sequence $S$, $\dchiS(G)$ is defined as $\dom_{\chi_{S'}}(G)$, where $S'=(s_1,\ldots,s_{\chi_S(G)})$ is the finite sequence consisting of the first $\chi_S(G)$ terms of $S$. Thus, we can write $S'=(1,2,\ldots,\chi_\rho(G))$. 

Let $f$ be an arbitrary $S'$-packing coloring of $G$. We will show that $f$ is a $3$-dominating broadcast by presenting a partition $\pi=\{D_1,D_2,D_3\}$ of $V(G)$ such that $f|_{D_i}$ is a dominating broadcast for all $i\in [3]$. 

We claim that we may assume that the set $A_1=\{v\in V(G):\, f(v)=1\}$ is a maximal independent set. Indeed, if $A_1$ is not a maximal independent set we may select an arbitrary maximal independent set $B_1$ such that $A_1\subsetneq B_1$ and for any $v\in B_1\setminus A_1$ replace the color of $v$ assigned by $f$ to color $1$. It is clear that the resulting coloring is still an $S'$-packing coloring. Now, under this assumption, letting $D_1=A_1$, we note that $D_1$ is an independent dominating set, and so $f|_{D_1}$ is a dominating broadcast, as desired. 

In addition to the above assumption, let $f$ be chosen in such a way that the set $A_2=\{v\in V(G):\, f(v)=2\}$ is as large as possible. We claim that, under this assumption, every vertex $x\in V(G)$ with $f(x)>2$ has a vertex $y\in A_2$ such that $d_G(x,y)\le 2$. Indeed, if a vertex $x\in V(G)$ with $f(x)>2$ has no vertices $y$ with $f(y)=2$ within distance $2$, then we can redefine $f$ by changing the value at $x$ to $2$, and the resulting function $f'$ is still an $S'$-packing coloring of $G$ with $A_1$ forming a dominating set, yet $|\{v:\,f'(v)=2\}|>|A_2|$, a contradiction.  

Let $$A_1''=\{v\in A_1:\,\textrm{ if }d_G(x,v)\le 2,\textrm{ then }f(x)\ne 2\}.$$
In other words, $A_1''$ are the vertices in $A_1$ that are not $f$-dominated by $A_2$. Let $A_1'=A_1\setminus A_1''$. If $A_1''\ne\emptyset$, then define $W$ as a smallest subset of $N_G(A_1'')$ such that every $x\in A_1''$ has a neighbor in $W$. Note that $W\cap A_2=\emptyset$. In addition, now that $W$ is defined, define a subset $U$ of vertices in $A_2$ as follows: $u\in U$ if $u\in A_2$, $N_G(u)\subset A_1$, and for every $v\in N_G(u)$ we have $N_G(v)\setminus \{u\}\subset W$. In other words, all neighbors of $u$ are in $D_1$ and their neighbors (with the obvious exception of $u$) belong to $W$. 

Let $D_2=(A_2\setminus U)\cup W$, and $D_3=(A_{\ge 3}\setminus W)\cup U$, where $$A_{\ge 3}=\{v\in V(G):\,f(v)\ge 3\}.$$  
Note that $\pi=\{D_1,D_2,D_3\}$ is a partition of $V(G)$. As mentioned earlier, $f|_{D_1}$ is a dominating broadcast, since $D_1$ is a dominating set. Let us now consider the vertices in $D_2$. 

\begin{claim}
\label{cl:0} $f|_{D_2}$ is a dominating broadcast.
\end{claim}
\noindent {\it Proof of Claim~\ref{cl:0}.}
As proved earlier, vertices of $A_2$ $f$-dominate vertices in $A_{\ge3}$. Hence, vertices in $A_2\setminus U$ dominate all vertices of $A_{\ge 3}$ with the possible exception of vertices in $A_{\ge 3}$ that are $f$-dominated by a vertex in $U$. However, by definition of a vertex  $u\in U$, among vertices in $A_{\ge 3}$ vertex $u$ only $f$-dominates a vertex $w\in W$, which is itself in $D_2$. Hence, vertices in $A_2\setminus U$, $f$-dominate all vertices in $A_{\ge 3}\setminus W$. In addition, vertices in $U$ are $f$-dominated by vertices in $W$. Combining these two findings we infer that vertices in $D_2=(A_2\setminus U)\cup W$ $f$-dominate all vertices in $(A_{\ge 3}\setminus W)\cup U=D_3$.

Vertices in $A_2\setminus U$ also $f$-dominate vertices in $A_1'$ by the definition $A_1'$, again with a possible exception of the vertices $R$ in $A_1'$ that are adjacent to a vertex $u\in U$; however, vertices in $R$ are dominated by a vertex $w\in W$, which is in $N_G(R)$. In addition, vertices of $W$ $f$-dominate vertices in $A_1''$. Hence, vertices in $D_2=(A_2\setminus U)\cup W$ $f$-dominate all vertices in $A_1$. Thus, $f|_{D_2}$ is a dominating broadcast. \hskip 0.6cm \smallqed

\medskip 

Next, let us consider $f|_{D_3}$; we divide the proof into two claims. 

\begin{claim}
\label{cl:1}
Every vertex in $A_1$ is $f$-dominated by a vertex in $D_3$.    
\end{claim}
\noindent {\it Proof of Claim~\ref{cl:1}.} Let $x\in A_1$. Note that $x$ has at least one neighbor $y$ in $A_{\ge 3}$. If $y\in D_3$, then $x$ is $f$-dominated by a vertex in $D_3$, and we are done. Thus, assume that $y\in W$. Hence, $y$ has a neighbor in $A_1''$ and let $Z=N_G(y)\cap A_1''$. Note that all neighbors of vertices in $Z$ belong to $A_{\ge 3}$. We claim that there exists at least one vertex $z\in Z$, which has a neighbor in $V(G)\setminus W$. 

Suppose to the contrary that for every vertex $z$ in $Z$ its neighborhood belongs to $W$. Then, $W\setminus \{y\}$ still has the property that every vertex in $A_1''$ has a neighbor in  $W\setminus \{y\}$. Indeed, since $\deg_G(z)\ge 2$ for each $z\in Z$, there exists at least one neighbor of $z$   that belongs to $W\setminus\{y\}$.  This is a contradiction to the minimality of $W$, thus our supposition was false.   

Hence there exists a vertex $z\in Z$, which has a neighbor $u$ in $V(G)\setminus W$, and note that $f(u)\ge 3$ since $z\in A_1''$. Thus, $u\in D_3$ and $$d_G(x,u)\le d_G(x,y)+d_G(y,u)\le 3\le f(u),$$ 
that is, $x$ is $f$-dominated by $u$.  \hskip 0.6cm \smallqed

\begin{claim}
\label{cl:2}
Every vertex in $D_2$ is $f$-dominated by a vertex in $D_3$.    
\end{claim}
\noindent {\it Proof of Claim~\ref{cl:2}.} Note that $D_2=(A_2\setminus U)\cup W$. 
The proof that every vertex in $W$ is $f$-dominated by a vertex in $D_3$ is similar as the second part of the proof of Claim~\ref{cl:1}. Notably, if $y\in W$, then $y$ has a neighbor in $A_1''$ and we again define $Z=N_G(y)\cap A_1''$. By following the same lines as in the above proof, we derive that there exists at least one vertex $z\in Z$, which has a neighbor $u$ in $V(G)\setminus W$. Since $u$ belongs to $A_{\ge 3}$ and $u\notin W$, we get $u\in D_3$. Clearly, $y$ is $f$-dominated by $u$, since it is at distance at most $2$ from $u$ with $f(u)\ge 3$. 

Now, consider a vertex $y\in A_2\setminus U$. If $y$ has a neighbor $x$ in $A_{\ge 3}$, then either $x\in D_3$ or $x\in W$. In the former case, $y$ is $f$-dominated by $x\in D_3$ and we are done. In the latter case, when $x\in W$ we derive in the same way as in the previous paragraph, that there is a vertex $u\in D_3$, which is at distance at most $2$ from $x$. Thus $u\in D_3$ $f$-dominates $y$, since $d_G(y,u)\le d_G(y,x)+d_G(x,u)\le 3\le f(u)$. 

Finally, assume that $y\in A_2\setminus U$ and $N_G(y)\subset D_1$. Let $T=N_G(y)$. Since $y\notin U$, there exists a vertex $t\in T$ such that $N_G(t)$ contains a vertex $x\notin W$. Since $d_G(y,x)=2$, we derive $f(x)\ge 3$, and consequently $x\in D_3$. Clearly, $x$ $f$-dominates $y$, since $d_G(y,x)=2< 3\le f(x)$. 

All in all, we proved that every vertex in $D_2$ is $f$-dominated by a vertex in $D_3$. \hskip 0.6cm \smallqed

\medskip 

Combining Claims~\ref{cl:1} and~\ref{cl:2} we derive that $f|_{D_3}$ is a dominating broadcast. Thus, $f$ is a packing $3$-domatic coloring of $G$, which implies $\dchir(G)\ge 3$.  
\end{proof}

We also see that the bound in Theorem~\ref{thm:domlower} is sharp. Indeed, if $n\equiv 0 \pmod 4$, we have $\chi_\rho(C_n)=3$, and by Proposition~\ref{3_cycles} we have $\chi_{\rho, 3}(C_n)=3$. Since by Proposition~\ref{4_cycle}, we have $\chi_{\rho,4}(C_8)=4$, we all in all infer that $\dchir(C_8)=3$. Similarly, Proposition~\ref{pro_20_cycle} provides $\chi_{\rho,4}(C_{20})=4$ and so $\dchir(C_{20})=3$. An infinite family of examples is obtained from the edge $xy$ by adding $\ell\ge 1$ common neighbors to vertices $x$ and $y$. The resulting graph $G_\ell$ has $\ell+2$ vertices, and note that $G_1$ is the graph $K_3$ and $G_2$ is the graph diamond $K_4-e$. In addition, $\chi_\rho(G_\ell)=3$. Indeed, letting $f(x)=2,f(y)=3$ and $f(u)=1$ for all vertices $u\notin\{x,y\}$ is a packing coloring of $G_\ell$, which is unique up to applying the symmetry of vertices $x$ and $y$. By Theorem~\ref{thm:domlower}, $\dchir(G_\ell)\ge 3$, and a set $A$, where $A\subseteq V(G_\ell)\setminus\{x,y\}$, is a dominating broadcast only if $A=V(G_{\ell})\setminus\{x,y\}$. Hence, $V(G_\ell)$ can be partitioned into (at most) three dominating broadcasts. Thus, $\dchir(G_\ell)=3$

Theorem~\ref{thm:domlower} is best possible also in the sense that the condition $\delta(G)\ge 2$ cannot be omitted. One can find many examples of graphs $G$ with $\delta(G)=1$ and $\dchir(G)=2$. A simple example are stars $K_{1,r}$ for which $\chi_\rho(K_{1,r})=2$. Another family of examples are graphs $G$ obtained from a double star by subdividing the edge between the two non-leaf vertices. In this case, $\chi_\rho(G)=3$ and again $\dchir(G)=2$. An example with $\chi_\rho(G)=4$ and again $\dchir(G)=2$ is depicted in Fig.~\ref{fig:exampleC}. A packing coloring with $4$ colors is depicted in Fig.~\ref{fig:exampleC}, while an alternative one can be obtained from it by swapping the colors $3$ and $2$ in the two vertices of the $5$-cycle and then adjust the colors in the corresponding ``tails'' that stem from these two vertices. In addition, one can modify the coloring in Fig.~\ref{fig:exampleC} by changing the color $1$ to color $2$ in one of the two twin leaves (the white vertices) adjacent to a vertex colored by color $3$. Regardless if this is done or not, a white leaf (with color $2$ if the color is changed) can only be $f$-dominated by two vertices.  Thus, indeed $\dchir(G)=2$.

\begin{figure}[h!]
\begin{center}
\begin{tikzpicture}[vertex/.style={circle, draw, fill=black, inner sep=2pt}]
    \node[vertex] (v1) at (90:2) {};
    \node[vertex] (v2) at (162:2) {};
    \node[vertex] (v3) at (234:2) {};
    \node[vertex] (v4) at (306:2) {};
    \node[vertex] (v5) at (18:2) {};
    \node[vertex] (v6) at (0,3.4) {};
     \node[vertex] (v7) at (1.4,2) {};
    \node[vertex] (v8) at (2.8,2) {};   
    \node[circle, draw, fill=white,inner sep=2.5pt] (v9) at (4.2,2) {};
    \node[circle, draw, fill=white,inner sep=2.5pt] (v10) at (3.7,2.6) {};
      \node[vertex] (x6) at (1.2,-3) {};
     \node[vertex] (x7) at (2.6,-1.6) {};
    \node[vertex] (x8) at (4,-1.6) {};   
    \node[vertex] (x9) at (5.4,-1.6) {};
    \node[vertex] (x10) at (4.9,-2.3) {};
    \node[vertex] (y6) at (-1.2,-3) {};
     \node[vertex] (y7) at (-2.6,-1.6) {};
    \node[vertex] (y8) at (-4,-1.6) {};   
    \node[vertex] (y9) at (-5.4,-1.6) {};
    \node[vertex] (y10) at (-4.9,-2.3) {};

    \draw[thick] (v1) -- (v2) -- (v3) -- (v4) -- (v5) -- (v1)--(v6);
     \draw[thick] (v1) -- (v7) -- (v8) -- (v9);
     \draw[thick] (v8) -- (v10);
      \draw[thick] (v4) -- (x7) -- (x8) -- (x9);
    \draw[thick] (x8) -- (x10);
    \draw[thick] (v4) -- (x6);
    \draw[thick] (v3) -- (y7) -- (y8) -- (y9);
    \draw[thick] (y8) -- (y10);
    \draw[thick] (v3) -- (y6);
     
    \node[above left] at (v1) {$2$};
    \node[left] at (v2) {$1$};
    \node[below left] at (v3) {$3$};
    \node[below right] at (v4) {$4$};
    \node[right] at (v5) {$1$};
      \node[above] at (v7) {$1$};
    \node[above] at (v8) {$3$};
    \node[right] at (v9) {$1$};
    \node[above] at (v10) {$1$};
     \node[left] at (v6) {$1$};
       \node[right] at (x6) {$1$};
       \node[above] at (x7) {$1$};
     \node[above] at (x9) {$1$};
      \node[above] at (x8) {$2$};
        \node[right] at (x10) {$1$};
         \node[left] at (y6) {$1$};
       \node[above] at (y7) {$1$};
     \node[above] at (y9) {$1$};
      \node[above] at (y8) {$2$};
        \node[left] at (y10) {$1$};
\end{tikzpicture}
\caption{A graph $G$ with $\chi_\rho(G)=4$ and $\dchir(G)=2$.}
\label{fig:exampleC}
\end{center}
\end{figure}

We remark that $\dchir(G)$ can be much bigger than $3$ and can, in fact, be arbitrarily large. Indeed, by Proposition~\ref{prp:diam2}, $\dom_{\chi_\rho}(G)=\beta(G)+1$ for any graph $G$ with ${\rm diam}(G)=2$.

Since the decision version of the packing chromatic number is NP-complete even for graphs of diameter $2$ and minimum degree at least $2$, we infer the following result from Theorem~\ref{thm:domlower} combined with Observations~\ref{ob:domandchi} and~\ref{obs:domatic}.

\begin{corollary}
\label{cor:completexity3}
The decision version of $\chi_{\rho,3}$ is NP-complete even when restricted to diameter-$2$ graphs with minimum degree at least $2$.  
\end{corollary}

\section{Concluding remarks}
\label{sec:conclude}

Results in Section~\ref{sec:paths} concerning finite paths show that $\chi_{\rho,k}(P_n)=k$ as soon as $n\ge 8$ (in fact, for $k\in\{3,4\}$ the equality holds even for some smaller $n$). On the other hand, by Theorem~\ref{theorem_chi_rho_k_12}, $\chi_{\rho,k}(P_{n})>k$ for any $k \geq 12$ and any $n \geq k$. Having in mind this discrepancy, the natural question is where this ``jump'' happens, and we formulate this problem as follows. 
\begin{problem}
\label{p:jump}
 What is the minimum integer $k$ such that $\chi_{\rho,k}(P_{n})>k$ for all $n\ge k$?
\end{problem}
\noindent Our preliminary study also showed that for $k=6$ and any $n \geq 10$, $\chi_{\rho,6}(P_n)=6$, hence the solution to Problem~\ref{p:jump} lies between $7$ and $12$.

Note that $\chi_{\rho,k}(P_n)\le k+1$ for all $k\in \{3,4,5\}$ and $n\ge k$. We have also checked that $\chi_{\rho,6}(P_6)=8$. However, if $n$ is large enough with respect to $k$, we think that $k+1$ could be a general upper bound for $\chi_{\rho,k}(P_n)$. 

\begin{problem}
Is it true that $\chi_{\rho,k}(P_n)\le k+1$ for any $k\ge 3$ and any $n\ge 2k$?
\end{problem}

\bigskip

Concerning the upper bound on $\chi_{\rho, k}(P_\infty)$ we suspect the following question has an affirmative answer. 

\begin{problem}
    Is it true that $\chi_{\rho, k}(P_\infty) \le k-1$ for any integer $k\ge 4$?
\end{problem}
\noindent Note that the above problem has the positive answer when $4\le k\le 6$. 
 
 Based on examples of graphs with $\dchir(G)=2$ presented in Section~\ref{sec:domatic}, it would be interesting to determine if there exists a family of graphs $G$ with $\dchir(G)=2$ and $\chi_\rho(G)$ arbitrarily large. In addition, in the light of Theorem~\ref{thm:domlower} it would be interesting to characterize the class of graphs with $\delta(G)\ge 2$ and $\dom_{\chi_\rho}(G)=3$. Combining the above, we pose the following problem.
 \begin{problem}
For $k\in\{2,3\}$, characterize the graphs $G$ with $\dchir(G)=k$. 
\end{problem}

Recall that due to Observation~\ref{obs:domatic} and Theorem~\ref{thm:domlower}, respectively, we have
$$\min\{\dchir(G):\,\delta(G)\ge 1\}=2, \textrm{ and} \qquad \min\{\dchir(G):\,\delta(G)\ge 2\}=3.$$
It would be interesting to study the corresponding behavior of $\dchir(G)$ in graphs $G$ with fixed minimum degree $k$ for $k\ge 3$. 
\begin{problem}
Determine $\min\{\dchir(G):\,\delta(G)\ge k \}$ for $k\ge 3$. 
\end{problem}
 
The main focus of this paper was on the packing domatic numbers of graphs, while the $S$-packing domatic numbers for sequences $S$ different from the sequence $(n)_{n\in \mathbb{N}}$ were considered only briefly in Section~\ref{sec:prel}. Hence, this large topic remains open for future research.

\section*{Acknowledgements}
B.B. and J.F. acknowledge the financial support of the Slovenian Research and Innovation Agency (research core funding No.\ P1-0297 and projects N1-0285, N1-0431 and J1-70045).
W.H. was supported by the National Natural Science Foundation of China (No.\ 12371345) and CSC scholarship program (No.\ 202508420004). DeepSeek assisted us in writing the code that appears in Appendix A.  Chat GPT-6.1 Sol was used for presentational purposes and for verifying some of the proofs; the authors take full responsibility for the final version of the paper.

\newpage 

\section*{Appendix A: $\chi_{\rho,4}(C_{20}) \geq 4$}

We verified that $\chi_{\rho,4}(C_{20}) \geq 4$ by computer search over all possible 4-partitions of $C_{20}$. First, notice that the only possible size $4$-tuples are $(4,4,4,8)$, $(4,4,5,7)$, $(4,4,6,6)$, $(4,5,5,6)$, and  $(5,5,5,5)$. The program examines each of these five cases and verifies that none admits a valid partition yielding a packing $4$-domatic coloring. 
We list it below. 

As a matter of interest, we also searched for set size combinations with three sets, and obtained the following outcome:

  \begin{table}[htbp]
	\centering
	\caption{Number of possible unions \(A_1 \cup A_2 \cup A_3\) for each size combination, where \(A_1, A_2, A_3\) are pairwise disjoint dominating broadcasts.}
	\label{tab:k3_unions}
	\begin{tabular}{cc}
		\toprule
		\((|A_1|, |A_2|, |A_3|)\) & Number of possible unions \\
		\midrule
		\((4, 4, 4)\) & 20 \\
		\((4, 4, 5)\) & 165 \\
		\((4, 4, 6)\) & 470 \\
		\((4, 5, 5)\) & 535 \\
		\((5, 5, 5)\) & 902 \\
		\bottomrule
	\end{tabular}
\end{table}

\bigskip 

\definecolor{lightgray}{HTML}{F5F5F5}
\definecolor{mystring}{RGB}{128,0,128} 

\lstset{
	language=Python,
	numbers=left,
	numberstyle=\tiny\color{gray},
	stepnumber=1,
	numbersep=5pt,
	basicstyle=\ttfamily\small,
	columns=fixed,
	keepspaces=true,
	tabsize=4,
	showspaces=false,
	showstringspaces=false,
	breaklines=true,
	frame=none,
    	backgroundcolor=\color{lightgray!20},
	keywordstyle=\color{blue},
	commentstyle=\color{green!60!black},
	stringstyle=\color{mystring},
	literate=
	{_cap_}{{\color{mystring}$\cap$}}1
	{_cup_}{{\color{mystring}$\cup$}}1
	{_rho_}{{\color{mystring}$\rho$}}1
	{_chi_}{{\color{mystring}$\chi$}}1
	{_emptyset_}{{\color{mystring}$\emptyset$}}1
	{_ge_}{{\color{mystring}$\ge$}}1,
}

\begin{lstlisting}
from itertools import combinations

n = 20
coloring = [3, 1, 2, 1] * 5  


def dist(i, j):
    return min(abs(i - j), n - abs(i - j))


def is_dom_broadcast(subset):
    if not subset:
        return False
    for v in range(n):
        if not any(dist(v, u) <= coloring[u] for u in subset):
            return False
    return True


# 2. Precompute all dominating broadcast subsets (sizes 4-8)
print("=" * 70)
print("Precomputing all dominating broadcast subsets")
print("=" * 70)

dom = {}
for size in [4, 5, 6, 7, 8]:
    dom[size] = []
    for subset in combinations(range(n), size):
        if is_dom_broadcast(list(subset)):
            dom[size].append(set(subset))
    print(f"  Size {size}: found {len(dom[size])}")

print()

# 3.List all unions of A1, A2 and A3 for k=3
print("=" * 70)
print("List all unions of these disjoint dominating broadcast sets A1, A2, A3")
print("Condition: - (A1_cap_A2) _cup_ (A1_cap_A3) _cup_ (A2_cap_A3) = _emptyset_")
print("=" * 70)
print()

size_combos_3 = [(4, 4, 4), (4, 4, 5), (4, 4, 6), (4, 5, 5), (5, 5, 5)]

for sizes in size_combos_3              
    print(f"Size combination {sizes}:")
    lists = [dom[s] for s in sizes]

    if any(len(lst) == 0 for lst in lists):
        print("  WARNING: Some size has no dominating subsets, skipping")
        print()
        continue

    union_set = set()
    count = 0
    max_display = 5

    for a1 in lists[0]:
        for a2 in lists[1]:
            if a1 & a2:
                continue
            for a3 in lists[2]:
                if (a1 & a3) or (a2 & a3):
                    continue
                union = frozenset(a1 | a2 | a3)
                if union not in union_set:
                    union_set.add(union)
                    count += 1
                    if count <= max_display:
                        print(f"  Union {count}: A1={sorted(a1)}, A2={sorted(a2)}, A3={sorted(a3)}")
                        print(f"           A1 _cup_ A2 _cup_ A3 = {sorted(union)}")
                        print(f"           |union| = {len(union)}")
                        print()

    print(f"  Total disjoint unions for this size combo: {count}")
    if count > max_display:
        print(f"  (Showing only first {max_display})")
    print()

# 4. k=4: Search for a 4-partition
print("=" * 70)
print("k=4: Search for 4-partition (A1 _cup_ A2 _cup_ A3 _cup_ A4 = V(C20))")
print("Conditions: Each Ai is a dominating broadcast, pairwise disjoint")
print("=" * 70)
print()

size_combos_4 = [(4, 4, 4, 8), (4, 4, 5, 7), (4, 4, 6, 6), (4, 5, 5, 6), (5, 5, 5, 5),
]

found4 = False

for sizes in size_combos_4:
    print(f"Trying size combination {sizes}:")
    lists = [dom[s] for s in sizes]

    if any(len(lst) == 0 for lst in lists):
        print("  WARNING: Some size has no dominating subsets, skipping")
        print()
        continue


    def backtrack(level, selected):
        if level == 4:
            union = set()
            for s in selected:
                union.update(s)
            return len(union) == n

        for subset in lists[level]:
            disjoint = True
            for s in selected:
                if subset & s:
                    disjoint = False
                    break
            if not disjoint:
                continue

            if backtrack(level + 1, selected + [subset]):
                return True
        return False


    if backtrack(0, []):
        found4 = True
        print(f"  FOUND partition!")
        break
    else:
        print("  NOT found")
        print()

# 5. Final conclusion
print("=" * 70)
print("Final Conclusion")
print("=" * 70)

if found4:
    print("WARNING: This 3-coloring CAN be partitioned into 4 dominating broadcast parts")
else:
    print("NOT found: 4-partition does not exist")
    print("Under all 5 size combinations, no 4 mutually disjoint")
    print("dominating broadcast subsets cover all vertices.")
    print()
    print("Therefore, this 3-coloring CANNOT be partitioned")
    print("into 4 dominating broadcast parts.")
    print()
    print("CONCLUSION:  The packing 4-domatic chromatic number of C20 is at least 4.")
print("=" * 70)
\end{lstlisting}

\end{document}